\documentclass[11pt,twoside]{article}
\usepackage{amsmath, amsthm, amsfonts, amssymb}
\usepackage{mathrsfs}
\usepackage[misc]{ifsym}
\usepackage[colorlinks,linkcolor=blue,anchorcolor=blue,citecolor=blue,urlcolor=black]{hyperref}
\usepackage{graphicx}
\usepackage{subfigure}
\usepackage{tikz}
\usepackage{enumerate}
\usepackage{anysize}
\usepackage{setspace}
\usepackage{float}
\usepackage{chngcntr}
\counterwithin{figure}{section}
\newtheorem{example}{Example}[section]
\numberwithin{equation}{section}
\numberwithin{figure}{section}

\usetikzlibrary{arrows,shapes,chains,3d}
\marginsize{3.5cm}{3cm}{3cm}{3cm}
\allowdisplaybreaks

\newfam\msbfam
\font\tenmsb=msbm5    \textfont\msbfam=\tenmsb \font\sevenmsb=msbm5
\scriptfont\msbfam=\sevenmsb \font\fivemsb=msbm5
\scriptscriptfont\msbfam=\fivemsb

\newfam\bigfam
\font\tenbig=msbm5 scaled \magstep2   \textfont\bigfam=\tenbig
\font\sevenbig=msbm7 scaled \magstep2 \scriptfont\bigfam=\sevenbig
\font\fivebig=msbm5 scaled \magstep2
\scriptscriptfont\bigfam=\fivebig

\newtheorem{theorem}{Theorem}[section]
\newtheorem{lemma}[theorem]{Lemma}
\newtheorem{definition}[theorem]{Definition}

\theoremstyle{definition}
\newtheorem{remark}{Remark}[section]

\begin{document}

\title{\bf Riesz transform associated with the fractional Fourier transform and applications  in image edge detection   
 \footnote{This work was partially supported by the National Natural Science Foundation of China (Nos. 12071197, 11701251  and 12071052), the Natural Science Foundation of Shandong Province (Nos.   ZR2019YQ04), a Simons Foundation Fellows Award (No. 819503) and a Simons Foundation Grant (No. 624733).}}

\author{\bf Zunwei Fu$^a$, Loukas Grafakos$^b$, Yan Lin$^c$, Yue Wu$^a$, Shuhui Yang$^c$  
\\  \it{\small{$^a$School of Mathematics and Statistics, Linyi University, Linyi 276000, China}}\\  \it{\small{$^b$Department of Mathematics, University of Missouri, Columbia MO 65211, USA}}\\  \it{\small{$^c$School of Science, China University of Mining and Technology, Beijing 100083,  China} }  }

\renewcommand{\thefootnote}{}

\date{}

\maketitle
\footnotetext{ Email addresses:   zwfu@mail.bnu.edu.cn(Zunwei Fu),   grafakosl@missouri.edu(Loukas Grafakos),  linyan@cumtb.edu.cn(Yan Lin),    wuyue@lyu.edu.cn(Yue Wu),   yang\_shu\_hui@163.\ com(Shuhui Yang)}
\begin{minipage}{13.5cm}
{\bf Abstract}

\quad \quad The fractional Hilbert transform was introduced by Zayed \cite[Zayed, 1998]{z2} and has been widely used in  signal processing. In view of is connection with the fractional Fourier transform, Chen, the first, second and fourth authors of this paper in \cite[Chen  et al., 2021]{cfgw} studied the fractional Hilbert transform and other fractional multiplier operators  on the real line. The present paper is concerned  with a natural extension of the fractional Hilbert transform to higher dimensions: this extension is the fractional Riesz transform which is defined  by multiplication which a suitable chirp function on the fractional Fourier transform side. In  addition to a thorough study of the fractional Riesz transforms, 
in this work we also   investigate the 
boundedness of    singular integral  operators with chirp functions on rotation invariant spaces, 
chirp Hardy spaces and their relation to chirp BMO spaces, as well as     applications of the theory of fractional multipliers in partial differential equations. Through numerical simulation, we provide   physical and geometric interpretations of   high-dimensional fractional multipliers. Finally, we present an  application of the fractional Riesz transforms in edge detection which verifies a 
hypothesis insinuated in \cite[Xu  et al.,  2016]{xxwqwy}.  In fact our  numerical implementation  confirms  that   amplitude, phase, and direction information can be simultaneously extracted  by controlling the order of the fractional Riesz transform. 
\medskip

{\bf Keywords} \quad fractional Fourier transform,  fractional Riesz transform,   edge detection, chirp Hardy space, fractional multiplier
   
\medskip
 
\medskip 

%{\bf Mathematics Subject Classification}\quad 42B25 $\cdot$  35S05 $\cdot$ 47G30

\end{minipage}
\\
\medskip
\begin{footnotesize}
\begin{spacing}{-1.0}
\vspace{-20pt}
  \tableofcontents
  \end{spacing}
\end{footnotesize}

\medskip

\medskip

%\\\\\\\\\\\\\\\Introduction\\\\\\\\\\\\

\section{Introduction}

\label{sect:intro}

\ \ \ \ One of the fundamental operators in Fourier analysis theory is the Hilbert transform 
$$
H(f)(x)=\frac{1}{\pi}\ {\rm p}.{\rm v}.\int_\mathbb{R} \frac{f(y)}{x-y}dy, %\ \ \ \ f\in S(\mathbb{R}),
$$
which   is a continuous analogue of  the
 conjugate Fourier series.  The  early studies of the Hilbert transform 
were based on   complex analysis methods but around the 1920s these were complemented and enriched  by real analysis techniques. The Hilbert transform, being the prototype of  singular integrals,  provided significant inspiration for the subsequent development of this subject. The work of Calder\'{o}n  and Zygmund  \cite{cz} in 1952  furnished extensions of singular integrals to $\mathbb{R}^n$. This theory has left a big impact in analysis in view of its many applications, especially in the field of partial differential equations. Nowadays,  singular integral  operators   are important tools in harmonic analysis but also find 
many applications in applied mathematics. For instance, the  Hilbert transform plays a 
fundamental role   in communication systems and digital signal processing systems, such as in filter, edge detection and modulation theory \cite{g3,h}. As the Hilbert transform   is given by convolution with 
the kernel $1/(\pi t)$ on the real line,  in signal processing it can be understood as the output of a linear time invariant system with an impulse response of $1/(\pi t)$.

The Fourier transform is   a   powerful tool in the analysis and processing of stationary signals.
\begin{definition}
We define the Fourier transform of a function $f$ in the Schwartz class $S(\mathbb{R}^n)$ by 
$$
\hat{f}(\xi)=\mathcal{F}(f)=\frac{1}{\left(\sqrt{2\pi}\right)^n}\int_{\mathbb{R}^n}f(x)e^{- ix\cdot\xi}dx.
$$
\end{definition}

In time-frequency analysis, the Hilbert transform is also known as a $\pi/2$-phase shifter. The Hilbert transform can be defined in terms of the Fourier transform as the following multiplier operator 
\begin{equation}\label{1.1}
\mathcal{F}(Hf)(x)=-i{\rm sgn}(x)\mathcal{F}(f)(x).
\end{equation}
It can be seen from (\ref{1.1}) that the Hilbert transform is a phase-shift converter that multiplies the positive frequency portion of the original signal by $-i$; in other words, it maintains  the same amplitude and shifts the phase by $-\pi/2$, while the negative frequency portion is shifted by $\pi/2$.

The Riesz transform is a generalization of the Hilbert transform in the $n$-dimensional case and is also a singular integral operator, with properties analogous to those of the Hilbert transform on $\mathbb{R}$. It is defined as 
\begin{eqnarray*}
R_j(f)(\boldsymbol{x})= c_n\ {\rm p}.{\rm v}.\int_{\mathbb{R}^n}\frac{x_j-y_j}{|\boldsymbol{x}-\boldsymbol{y}|^{n+1}}f(\boldsymbol{y})d\boldsymbol{y}, \ \ 1\leq j\leq n,
\end{eqnarray*}
where $c_n= {\Gamma(\frac{n+1}{2})}/{\pi^{\frac{n+1}{2}}} $.
The Riesz transform is also a multiplier operator $$\mathcal{F}(R_jf)(\boldsymbol{x})=-\frac{ix_j}{|\boldsymbol{x}|}\mathcal{F}(f)(\boldsymbol{x}).$$
\begin{remark}
The multiplier of the Hilbert transform is $-i{\rm sgn}(x)$, and it is simply a phase-shift converter. The multiplier of the Riesz transform is $- ix_j/|\boldsymbol{x}|$, and thus, the Riesz transform is not only a phase-shift converter but also an amplitude attenuator.
\end{remark}
The Riesz transform  has wide applications in image edge detection, image quality assessment and biometric feature recognition \cite{la,zl,zzm}.

 The Fourier transform is     limited  in processing and analyzing  nonstationary signals. The fractional Fourier transform (FRFT) was proposed and developed by some scholars mainly because of the need for  nonstationary signals. The FRFT originated in the work of Wiener in \cite{w2}.  Namias in \cite{n} proposed the FRFT through a   method that was primarily based on eigenfunction expansions in 1980. McBride-Kerr in \cite{mk} and Kerr in \cite{k} provided   integral expressions of the FRFT on $S(\mathbb{R} )$ and $L^2(\mathbb{R} )$, respectively.    In \cite{cfgw},   Chen, and  the first, second and fourth authors of this paper, established   the behavior of FRFT on $L^p(\mathbb{R} )$ for  $1\leq p<2$.

A chirp function  is a nonstationary signal in which the frequency increases (upchirp) or decreases (downchirp) with time. The chirp signal is the most common nonstationary signal. In  1998,  Zayed in \cite{z2} gave the following definition of the fractional Hilbert transform 
\begin{align*}
H_\alpha(f)(x)=  \frac{1}{\pi}\ {\rm p}.{\rm v}.\ e_{-\alpha}(x)  \int_\mathbb{R} \frac{f(y)}{x-y}e_\alpha(y)dy,
\end{align*}
where $e_\alpha(x)=e^{\frac{ix^2\cot\alpha}{2}} $ is a chirp function. 
%This operator can also be expressed in terms of  the   FRFT.  

In \cite{py},  Pei and Yeh expressed the discrete fractional Hilbert transform as a composition of the discrete fractional Fourier transform (DFRFT), 
a multiplier, %masking DFRFT 
and the inverse DFRFT; based on this they conducted simulation verification on the edge detection of digital images.  In \cite{cfgw}, Chen, and the first, second and fourth authors of this paper 
related  the fractional Hilbert transform to the fractional Fourier   multiplier
$$
\mathcal{F}_{\alpha}(H_{\alpha} f)(x)=-i{\rm sgn}((\pi-\alpha)x)\mathcal{F}_{\alpha}( f)(x),
$$
where $\mathcal{F}_{\alpha}$ is FRFT; see Definition \ref{de1.1}.
In analogy with the Hilbert transform, the fractional Hilbert transform is  also a phase-shift converter. As 
indicated above,   the continuous fractional Hilbert transform can be decomposed into a composition of 
the FRFT, a multiplier, and the inverse FRFT. % this confirms the prediction in   \cite{py}.
The fractional Hilbert transform can  also  be used in single sideband communication systems and image encryption systems. The rotation angle can be used as the encryption key to improve the communication security and image encryption effect in \cite{tlw}.

The  multidimensional FRFT has recently made its appearance:   Zayed   \cite{z3,z4} introduced a new two-dimensional FRFT. In \cite{kr},  Kamalakkannan  and  Roopkumar introduced the the multidimensional FRFT.
\begin{definition}\label{de1.1}\emph{(\cite{kr})}
The  multidimensional FRFT with order $\boldsymbol{\alpha}=(\alpha_1,\alpha_2,\dots,\alpha_n)$  on $L^1(\mathbb{R}^n)$ is defined by
$$\mathcal{F}_{\boldsymbol{\alpha}}(  f)(\boldsymbol{u})=\int_{\mathbb{R}^n}f(\boldsymbol{x})K_{\boldsymbol{\alpha}}(\boldsymbol{x},\boldsymbol{u})d\boldsymbol{x},$$
where $K_{\boldsymbol{\alpha}}(\boldsymbol{x},\boldsymbol{u}) =\prod^n_{k=1}K_{\alpha_k}(x_k,u_k)$ and $K_{\alpha_k}(x_k,u_k)$ are given by
\begin{equation*}
K_{\alpha_k}(x_k,u_k)=\left\{
\begin{array}
[c]{ll}
\frac{c(\alpha_k)}{\sqrt{2\pi}}e^{i(a(\alpha_k)(x_k^2+u_k^2-2b(\alpha_k)x_ku_k))},  &\alpha_k\notin\pi \mathbb{Z} ,\\
\delta(x_k-u_k), &  \alpha_k\in2\pi \mathbb{Z} ,\\
\delta(x_k+u_k),  &    \alpha_k\in2\pi\mathbb{Z}+\pi ,
\end{array}
\right.
\end{equation*}
 $\boldsymbol{x}=(x_1,x_2,\dots,x_n)$, $a(\alpha_k)=\frac{\cot(\alpha_k)}{2}$, $b(\alpha_k)=\sec( \alpha_k )$, $c(\alpha_k)=\sqrt{1-i\cot(\alpha_k)}$.
\end{definition}
\begin{remark}
Suppose that $\boldsymbol{\alpha}=(\alpha_1,\alpha_2,\dots,\alpha_n)\in \mathbb{R}^n$ with $\alpha_k\notin \pi \mathbb{Z}$ for all $ k=1,2,\dots,n $. Consider the chirps 
$$e_{\boldsymbol{\alpha}}(\boldsymbol{x})=e^{i\sum^n_{k=1}a(\alpha_k)x^2_k}, $$
 for    $\boldsymbol{x}\in \mathbb{R}^n.$
It is straightforward to observe that FRFT of $f$ can be written as
$$\mathcal{F}_{\boldsymbol{\alpha}}(  f)(\boldsymbol{u})=c(\boldsymbol{\alpha} )e_{\boldsymbol{\alpha}}(\boldsymbol{u})\mathcal{F}(e_{\boldsymbol{\alpha}} f)(\tilde{\boldsymbol{u}}),$$
where  $c(\boldsymbol{\alpha} )=c(\alpha_1 )\cdots c(\alpha_n )$,  $\tilde{\boldsymbol{u}}=(u_1\csc\alpha_1,\dots,u_n\csc\alpha_n)$. From the preceding identity, it can be seen that $\mathcal{F}_{\boldsymbol{\alpha}} $ 
is bounded from $S(\mathbb{R}^n)$ to $ S(\mathbb{R}^n)$. We  rewrite
$$K_{\boldsymbol{\alpha }}(\boldsymbol{x} ,\boldsymbol{u} )=\frac{c(\boldsymbol{\alpha} )}{\left(\sqrt{2\pi}\right)^{n}}e_{\boldsymbol{\alpha}}(\boldsymbol{x})e_{\boldsymbol{\alpha}}(\boldsymbol{u})e^{-i\sum^{n}_{k=1}x_ku_k\csc\alpha_k}.$$
\end{remark}
Motivated by this work,  we define the fractional Riesz transform  associated with  the multidimensional FRFT as follows:
\begin{definition} For $1\leq j\leq n$, the $j$th fractional Riesz transform of $f\in S(\mathbb{R}^n)$ is given by
\begin{align*}
R_j^{\boldsymbol{\alpha}}(f)(\boldsymbol{x})=c_n \ {\rm p}.{\rm v}. \ e_{-\boldsymbol{\alpha}}(\boldsymbol{x})\int_{\mathbb{R}^n}\frac{x_j-y_j}{|\boldsymbol{x}-\boldsymbol{y}|^{n+1}}f(\boldsymbol{y})e_{ \boldsymbol{\alpha}}(\boldsymbol{y})d\boldsymbol{y},
\end{align*}
where $c_n=\Gamma(\frac{n+1}{2})/\pi^{\frac{n+1}{2}}$  and $\boldsymbol{\alpha}=(\alpha_1,\dots,\alpha_n)\in \mathbb{R}^n$ with $\alpha_k\notin \pi \mathbb{Z},\  k=1,2,\dots,n$.
\end{definition}
\begin{remark}
The fractional Riesz transform  reduces to  the fractional Hilbert transform for   $n=1$, while the fractional Riesz transform   reduces to  the classical Riesz transform for  $\boldsymbol{\alpha}=(\frac{\pi}{2}+k_1\pi,\frac{\pi}{2}+k_2\pi,\dots, \frac{\pi}{2}+k_n\pi),\  k_j\in \mathbb{Z} $,  $j=1,2,\dots,m$.
\end{remark}
This paper will be organized as follows. In Section 2, we obtain characterizations of    the fractional Riesz transform in terms of the FRFT and we note  that the  fractional Riesz transform is not only a phase shift converter but also an amplitude attenuator.  We obtain the  identity   $\sum^n_{j=1} (R_j^{\boldsymbol{\alpha}} )^2=-I$ and the boundedness of singular integral operators with a chirp function  on rotation invariant spaces. In Section 3, we introduce the definition of the chirp Hardy space by taking the Possion maximum for the function with the chirp factor and study  the dual spaces of chirp Hardy spaces.  We also  characterize   the boundedness of  singular integral operators with chirp functions on chirp Hardy spaces.    In Section 4, we derive a formula for  the high-dimensional FRFT and we provide an application of the fractional Riesz transform  
to partial differential equations. In Section 5, we conduct   a simulation experiment   with the fractional Riesz transform on an image and give the physical and geometric interpretation of the high-dimensional fractional multiplier theorem. In Section 6, we discuss a situation where  it is difficult to directly use the fractional Riesz transform for edge detection but the fractional multiplier theorem provides this possibility. The use of the fractional Riesz transform is completely equivalent to the compound operation of the FRFT, inverse FRFT and multiplier, and the FRFT and inverse FRFT can realize fast operations.
\section{Fractional Riesz transforms}

\label{sect:L1}

 \subsection{Properties of the fractional Riesz transforms}
 
\begin{theorem}\label{2th:1}
The $j$th fractional Riesz transform $R_j^{\boldsymbol{\alpha}}$ is given on the FRFT side by multiplication by the function $-i\frac{\tilde{u}_j}{|\tilde{\boldsymbol{u}}|}$. That is,  for any  $f\in S(\mathbb{R}^n)$  we have
$$\mathcal{F}_{\boldsymbol{\alpha}}\left(R_j^{\boldsymbol{\alpha}} f\right)(\boldsymbol{u})=-i\frac{\tilde{u}_j}{|\tilde{\boldsymbol{u}}|}\mathcal{F}_{\boldsymbol{\alpha}}\left( f\right)(\boldsymbol{u}),$$
where $ \boldsymbol{\alpha} =(\alpha_1 ,\dots, \alpha_n )\in \mathbb{R}^n$ with $\alpha_k\notin \pi \mathbb{Z}, \ k=1,2,\dots,n$; $ \boldsymbol{u} =(u_1 ,\dots, u_n )$ and $\tilde{\boldsymbol{u}}=(u_1\csc\alpha_1,\dots,u_n\csc\alpha_n)=(\tilde{u}_1 ,\dots,\tilde{u}_n )$.
\end{theorem}
\begin{proof}
Fix a   $f\in S (\mathbb{R}^n)$.  For $1\leq j \leq n$,  we have
\begin{align*}
\mathcal{F}_{\boldsymbol{\alpha}}\left(R_j^{\boldsymbol{\alpha}} f\right)(\boldsymbol{u})=&\  \int_{\mathbb{R}^n}R_j^{\boldsymbol{\alpha}} f (\boldsymbol{x})\frac{c(\boldsymbol{\alpha})}{\left(\sqrt{2\pi}\right)^n}e_{\boldsymbol{\alpha}}(\boldsymbol{x})e_{\boldsymbol{\alpha}}(\boldsymbol{u})e^{-i\sum^{n}_{j=1}x_ju_j\csc{\alpha_j}}d\boldsymbol{x}\\
=&\ \int_{\mathbb{R}^n}\frac{\Gamma(\frac{n+1}{2})}{\pi^{\frac{n+1}{2}}}e_{-\boldsymbol{\alpha}}(\boldsymbol{x)}
\lim_{\varepsilon\rightarrow0}\int_{|\boldsymbol{y}|\geq\varepsilon}\frac{y_j}{|\boldsymbol{y}|^{n+1}}f(\boldsymbol{x}-\boldsymbol{y})e_{\boldsymbol{\alpha}}(\boldsymbol{x}-\boldsymbol{y})d\boldsymbol{y}\\
&\ \ \ \ \ \ \ \ \ \ \ \ \ \ \ \ \ \ \   \times \frac{c(\boldsymbol{\alpha)}}{\left( \sqrt{2\pi}\right)^n} e_{\boldsymbol{\alpha}}(\boldsymbol{x})e_{\boldsymbol{\alpha}}(\boldsymbol{u}) e^{-i\sum^{n}_{j=1}x_ju_j\csc{\alpha_j}}d\boldsymbol{x}\\
=& \ \frac{\Gamma(\frac{n+1}{2})}{\pi^{\frac{n+1}{2}}}\lim_{\varepsilon\rightarrow0}\int_{|\boldsymbol{y}|\geq\varepsilon}\frac{y_j}{|\boldsymbol{y}|^{n+1}}\int_{\mathbb{R}^n} f(\boldsymbol{x}-\boldsymbol{y})e_{\boldsymbol{\alpha}}(\boldsymbol{x}-\boldsymbol{y})\frac{c(\boldsymbol{\alpha})}{\left(\sqrt{2\pi}\right)^n} e_\alpha(\boldsymbol{u})\\
&\ \ \ \ \ \ \ \ \ \ \ \ \ \ \ \ \ \ \   \times  e^{-i\sum^{n}_{j=1}x_ju_j\csc{\alpha_j}}d\boldsymbol{x} d\boldsymbol{y}.
\end{align*}
From the  substitution of variables, we have
\begin{align*}
\mathcal{F}_{\boldsymbol{\alpha}}\left(R_j^{\boldsymbol{\alpha}} f\right)(\boldsymbol{u})=&\  \frac{\Gamma(\frac{n+1}{2})}{\pi^{\frac{n+1}{2}}}\lim_{\varepsilon\rightarrow0}\int_{|\boldsymbol{y}|\geq\varepsilon}\frac{y_j}{|\boldsymbol{y}|^{n+1}}\int_{\mathbb{R}^n} f(\boldsymbol{w})e_{\boldsymbol{\alpha}}(\boldsymbol{w})\frac{c(\boldsymbol{\alpha})}{\left(\sqrt{2\pi}\right)^n} e_{\boldsymbol{\alpha}}(\boldsymbol{u})\\
&\ \ \ \ \ \ \ \ \ \ \ \ \ \ \ \ \ \ \  \times  e^{-i\sum^{n}_{j=1}(y_j+w_j)u_j\csc{\alpha_j}}d\boldsymbol{w} d\boldsymbol{y}\\
=&\  \frac{\Gamma(\frac{n+1}{2})}{\pi^{\frac{n+1}{2}}}\lim_{\varepsilon\rightarrow0}\int_{|\boldsymbol{y}|\geq\varepsilon}\frac{y_j}{|\boldsymbol{y}|^{n+1}}e^{-i\sum^{n}_{j=1} y_j u_j\csc{\alpha_j}}d\boldsymbol{y}\int_{\mathbb{R}^n}\frac{c(\alpha)}{\left(\sqrt{2\pi}\right)^n}\\
&\ \ \ \ \ \ \ \ \ \ \ \ \ \ \ \ \ \ \   \times f(\boldsymbol{w})e_{\boldsymbol{\alpha}}(\boldsymbol{u})e_{\boldsymbol{\alpha}}(\boldsymbol{w}) e^{-i\sum^{n}_{j=1} w_j u_j\csc{\alpha_j}}d\boldsymbol{w}\\
=&\  \frac{\Gamma(\frac{n+1}{2})}{\pi^{\frac{n+1}{2}}}\mathcal{F}_{\boldsymbol{\alpha}}( f)(\boldsymbol{u})\lim_{\varepsilon\rightarrow0}\int_{|\boldsymbol{y}|\geq\varepsilon}\frac{y_j}{|\boldsymbol{y}|^{n+1}}e^{-i\sum^{n}_{j=1} y_j u_j\csc{\alpha_j}}d\boldsymbol{y}\\
=&\  \frac{\Gamma(\frac{n+1}{2})}{\pi^{\frac{n+1}{2}}}\mathcal{F}_{\boldsymbol{\alpha}}( f)(\boldsymbol{u})\lim_{\varepsilon\rightarrow0}\int_{\varepsilon\leq|\boldsymbol{y}|\leq\frac{1}{\varepsilon}}\frac{y_j}{|\boldsymbol{y}|^{n+1}}e^{-i  \boldsymbol{y}\cdot\tilde{\boldsymbol{u}}  }d\boldsymbol{y}.
\end{align*}
Switching to polar coordinates and using Lemma 5.1.15 in \cite{g1}, we obtain
\begin{align*}
\mathcal{F}_{\boldsymbol{\alpha}}\left(R_j^{\boldsymbol{\alpha}} f\right)(\boldsymbol{u})=&\  \frac{-i\Gamma(\frac{n+1}{2})}{\pi^{\frac{n+1}{2}}}\mathcal{F}_{\boldsymbol{\alpha}}( f)(\boldsymbol{u})\lim_{\varepsilon\rightarrow0}\int_{s^{n-1}}\int_{\varepsilon\leq r\leq\frac{1}{\varepsilon}}\sin(r\tilde{\boldsymbol{u}}\cdot\boldsymbol{\theta} )\frac{r}{r^{n+1}}r^{n-1}dr\theta_jd\boldsymbol{\theta} \\
=&\  \frac{-i\Gamma(\frac{n+1}{2})}{\pi^{\frac{n+1}{2}}}\mathcal{F}_{\boldsymbol{\alpha}}( f)(\boldsymbol{u}) \int_{s^{n-1}}\int_0^\infty \sin(r\tilde{\boldsymbol{u}}\cdot\boldsymbol{\theta} ) \frac{dr}{r}\theta_jd\boldsymbol{\theta}\\
=&\  \frac{-i\Gamma(\frac{n+1}{2})}{2\pi^{\frac{n-1}{2}}} \mathcal{F}_{\boldsymbol{\alpha}}( f)(\boldsymbol{u}) \int_{s^{n-1}} {\rm sgn}(\tilde{\boldsymbol{u}}\cdot\boldsymbol{\theta})\theta_jd\boldsymbol{\theta}\\
=&\  -i\frac{\tilde{u}_j}{|\tilde{\boldsymbol{u}}|}\  \mathcal{F}_{\boldsymbol{\alpha}}( f)(\boldsymbol{u}),
\end{align*}
which completes the proof of the theorem.
\end{proof}

\begin{lemma}\label{4th:4}(FRFT inversion theorem) \ \emph{(\cite{kr})}
Suppose $f\in S(\mathbb{R}^n)$. Then
$$f(\boldsymbol{x})=\int_{\mathbb{R}^n}\mathcal{F}_{\boldsymbol{\alpha}}(  f)(\boldsymbol{u})K_{-\boldsymbol{\alpha}}(\boldsymbol{u},\boldsymbol{x})d\boldsymbol{u},\ \ a.e.\ x\in\mathbb{R}^n.$$
\end{lemma}

By Theorem \ref{2th:1} and Lemma \ref{4th:4}, the $j$th fractional Riesz transform of order
$\boldsymbol{\alpha}$ can be rewritten as
\[
(R_j^{\boldsymbol{\alpha}}f) (\boldsymbol{x}) =\left[ \mathcal{F}_{-\boldsymbol{\alpha}}\left( -i
\frac{\tilde{u}_j}{|\tilde{\boldsymbol{u}}|} (\mathcal{F}_{\boldsymbol{\alpha}} f) (\boldsymbol{u}) \right)\right] (\boldsymbol{x}).
\]
Denote  $m_j^{\boldsymbol{\alpha}} ({\boldsymbol{u}}):=  -i {\tilde{u}_j}/{|\tilde{\boldsymbol{u}}|} $.
It can be seen  that the fractional Riesz transform of a function $f$ can be
decomposed into three simpler operators, according to the diagram of Fig.
\ref{fig:com}:

\begin{enumerate}
	[(i)]
	
	\item FRFT of order $\boldsymbol{\alpha}$, $g({\boldsymbol{u}}) =
	(\mathcal{F}_{\boldsymbol{\alpha}} f)({\boldsymbol{u}})$;
	
	\item multiplication by a fractional $L^p$ multiplier, $h({\boldsymbol{u}})
	= m_j^{\boldsymbol{\alpha}}({\boldsymbol{u}}) g({\boldsymbol{u}})$;
	
	\item FRFT of order $-\boldsymbol{\alpha}$,
	$(R_j^{\boldsymbol{\alpha} }f) (\boldsymbol{x}) = 	(\mathcal{F}_{-\boldsymbol{\alpha}} h)(\boldsymbol{x})$
\end{enumerate}

\begin{figure}[H]
	\centering
		\begin{tikzpicture}[thick]
			\node (start){$f(\boldsymbol{x})$};
			\node[node distance=18mm, rectangle,draw,right of=start]
			(FRFT){$\mathcal{F}_{\boldsymbol{\alpha}}$};
			\node[node distance=18mm, inner sep=0pt,right of=FRFT]
			(MA){$\bigotimes$};
			\node[node distance=14mm, below of =MA] (p){$m_j^{\boldsymbol{\alpha}} (\boldsymbol{u})$};
			\node[node distance=18mm, rectangle,draw,right of=MA]
			(iFRFT){$\mathcal{F}_{-\boldsymbol{\alpha}}$};
			\node[node distance=24mm, right of=iFRFT] (end){$(R_j^{\boldsymbol{\alpha} }f)
			(\boldsymbol{x})$};
			\draw[->](start)--(FRFT);
			\draw[->](FRFT)--node[above]{$g(\boldsymbol{u})$}(MA);
			\draw[->](MA)--node[above]{$h(\boldsymbol{u})$}(iFRFT);
			\draw[->](iFRFT)--(end);
			\draw[->](p)--(MA);
		\end{tikzpicture}
	\caption{The decomposition of the $j$th fractional Riesz transform.}%
	\label{fig:com}%
\end{figure}
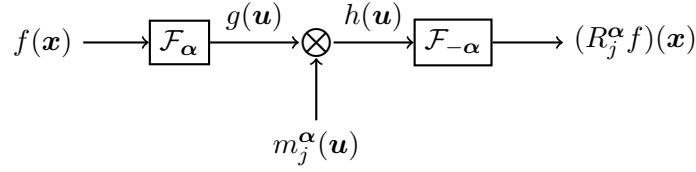

Take a 2-dimensional fractional Riesz transform as an example. It can be seen from Theorem \ref{2th:1}
that the fractional Riesz transform of order $\boldsymbol{\alpha}$ is a phase-shift converter that multiplies the positive portion in the $\boldsymbol{\alpha}$-order
fractional Fourier domain of signal $f$ by $-i$; in other words,  it shifts the phase by $-\pi/2$ while the
negative portion of $\mathcal{F}_{\boldsymbol{\alpha}}f$ is shifted by $\pi/2$. It is also an
amplitude reducer that multiplies the amplitude  in the  $\boldsymbol{\alpha}$-order fractional Fourier
domain of signal $f$ by  ${\tilde{u}_j}/{|\tilde{\boldsymbol{u}}|} $, as shown in Fig. \ref{fig:HT}.

\begin{figure}[H]
	\centering{\footnotesize
		\subfigure[]{
			\begin{tikzpicture}[>=stealth,thick,scale=0.7]			
				\draw [->] (0,-3,0) -- (0,3,0) node [at end, above] {Im $U$};
				\draw [->] (0,0,-3) -- (0,0,3) node [at end, below] {Re $U$};
				\filldraw [blue!20]  (1,0,0) -- (1,0,1.5) -- (2.5,0,1.5)--(2.5,0,0);
				\draw [densely dotted, draw=blue!70!red] (2.5,0,0)--(2.5,0,1.5);
				\draw [densely dotted, draw=blue!70!red]  (1,0,0) -- (1,0,1.5);
				\draw [draw=blue!70!red]  (1,0,1.5) -- (2.5,0,1.5);
				\filldraw [blue!50!green!20] (-1,0,0) -- (-1,0,1.5) -- (-2.5,0,1.5)--(-2.5,0,0);
				\draw [densely dotted, draw=blue!50!green] 	(-2.5,0,0)--(-2.5,0,1.5);
				\draw [densely dotted, draw=blue!50!green]  (-1,0,0) -- (-1,0,1.5);
				\draw [draw=blue!50!green]  (-1,0,1.5) -- (-2.5,0,1.5);
				\draw [->] (-3,0,0) -- (3,0,0) node [at end, right] {$\boldsymbol{u}$};
			\end{tikzpicture}
		}
		%%%%%%%%%%%%%%%%%%%%%%%%%%%%%%%%%%%%%%%%%%%%%%%%%%%%%%%%%%%%%%%%%%%%%%%%%%%%%%%%%
		\subfigure[]{
			\begin{tikzpicture}[>=stealth,thick,scale=0.7]
				\draw [->] (0,-3,0) -- (0,3,0) node [at end, above] {Im $V$};
				\draw [->] (0,0,-3) -- (0,0,3) node [at end, below] {Re $V$};
				\draw [densely dotted,fill=blue!20,draw=blue!70!red,rotate around x=90]
				(1,0,0) -- (1,0,1) -- (2.5,0,1)--(2.5,0,0);
				\draw[draw=blue!70!red,rotate around x=-90]
				(1,0,-0.3) .. controls (1.3, 0, -0.6) and (1.7,0,-0.8) .. (2.5,0, -0.8);
				\draw [densely dotted,draw=blue!50!green,fill=blue!50!green!20,rotate around x=-90]
				(-1,0,0) -- (-1,0,1) -- (-2.5,0,1)--(-2.5,0,0);
				\draw[draw=blue!50!green,rotate around x=-90]
				(-1,0,0.3) .. controls (-1.3, 0, 0.6) and (-1.7,0,0.8) .. (-2.5,0, 0.8);
				\draw [->] (-3,0,0) -- (3,0,0) node [at end, right] {$\boldsymbol{u}$};
			\end{tikzpicture}
		}\\
		%%%%%%%%%%%%%%%%%%%%%%%%%%%%%%%%%%%%%%%%%%%%%%%%%%%%%%%%%%%%%%%%%%%%%%%%%%%%%%%%
		\subfigure[]{
			\begin{tikzpicture}[>=stealth,thick,scale=0.7]
				\draw[->](-3,0)--(3,0) node[right]{$x_1$};
				\draw[->](0,-3)--(0,3)node[above]{$\omega_1$};
				\draw[->,rotate=60,blue](-3,0)--(3,0) node[right]{$u_1$};
				\draw[->,thin,red] (1.5,0) arc (0:60:1.5);
				\draw[->,thin] (0.6,0) arc (0:90:0.6);
				\node[right] at (1.2,0.9)
				{\footnotesize$\textcolor{red}{\mathcal{F}_{\alpha_1}}$};
				\node[right] at (0.4,0.4) {\footnotesize$\mathcal{F}_{\frac{\pi}{2}}$};
			\end{tikzpicture}
		}
	%%%%%%%%%%%%%%%%%%%%%%%%%%%%%%%%%%%%%%%%%%%%%%%%%%%%%%%%%%%%%%%%%%%%%%%%%%%%%%%%
		\subfigure[]{
			\begin{tikzpicture}[>=stealth,thick,scale=0.7]
				\draw[->](-3,0)--(3,0) node[right]{$x_2$};
				\draw[->](0,-3)--(0,3)node[above]{$\omega_2$};
				\draw[->,rotate=45,blue](-3,0)--(3,0) node[right]{$u_2$};
				\draw[->,thin,red] (1.5,0) arc (0:45:1.5);
				\draw[->,thin] (0.6,0) arc (0:90:0.6);
				\node[right] at (1.4,0.7)
				{\footnotesize$\textcolor{red}{\mathcal{F}_{\alpha_2}}$};
				\node[right] at (0.4,0.4) {\footnotesize$\mathcal{F}_{\frac{\pi}{2}}$};
			\end{tikzpicture}
		}
		%%%%%%%%%%%%%%%%%%%%%%%%%%%%%%%%%%%%%%
	}
	\caption{(a) the original signal: $U=(\mathcal{F}_{\boldsymbol{\alpha}}f)(\boldsymbol{u})$;
		(b) after fractional Riesz transform of order $\boldsymbol{\alpha}$:
		$V=(\mathcal{F}_{\boldsymbol{\alpha}} (R_j^{\alpha }f) (\boldsymbol{u})$;
		(c)-(d) rotations of the time-frequency planes, $\boldsymbol{u}=(u_1,u_2)$,
		$\boldsymbol{x}=(x_1,x_2)$, $\boldsymbol{\alpha}=(\alpha_1,\alpha_2)$.}%
	\label{fig:HT}%
\end{figure}
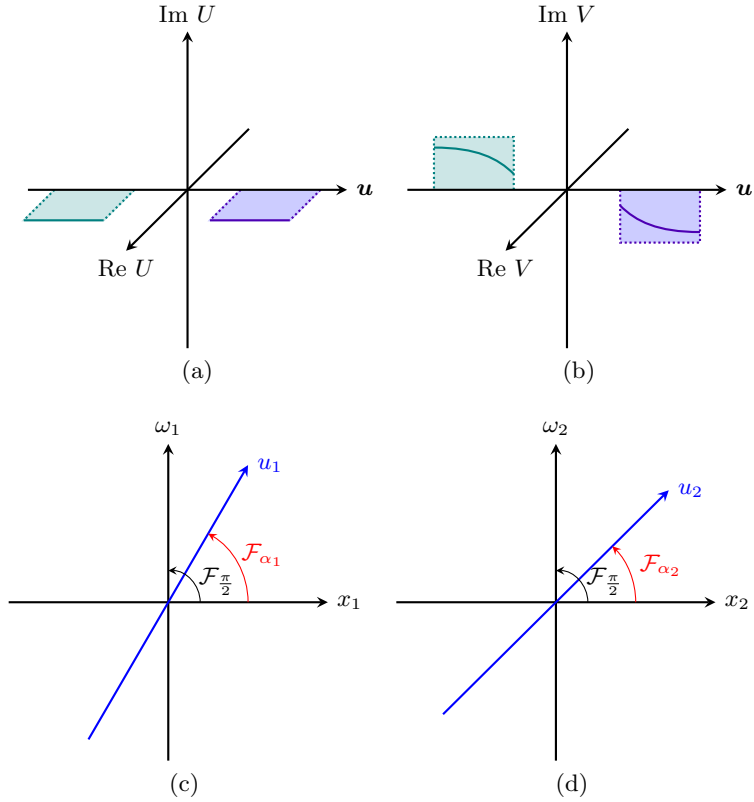

Next, we establish the $L^p(\mathbb{R}^n)$ boundedness of the fractional Riesz transform.
\begin{theorem}\label{2th:6}
For all $1<p<\infty$, there exists  a positive constant $C$ such that
$$\left\|R_j^{\boldsymbol{\alpha}}(f)\right\|_{L^p}\leq C\left\| f \right\|_{L^p},$$
for all $f$ in $  S (\mathbb{R}^n)$.
\end{theorem}

\begin{proof}
From the $L^p$ boundedness of the   Riesz transform in \cite{ldy} with Theorem 2.1.4, it follows that
 \begin{align*}
\left\|R_j^{\boldsymbol{\alpha}}(f)\right\|_{L^p} =&\ \left(\int_{\mathbb{R}^n}\Big|c_n e_{-\boldsymbol{\alpha}}(\boldsymbol{x})\int_{\mathbb{R}^n}\frac{y_j}{|\boldsymbol{y}|^{n+1}}f(\boldsymbol{x}-\boldsymbol{y})e_{\boldsymbol{\alpha}}(\boldsymbol{x}-\boldsymbol{y})d\boldsymbol{y}\Big|^pd\boldsymbol{x}\right)^{\frac{1}{p}}\\
=&\ \left\|R_j (fe_{\boldsymbol{\alpha}})\right\|_{L^p}\\
\leq&\  C\|   f \|_{L^p},
\end{align*}
for all $f$ in $  S (\mathbb{R}^n)$.
\end{proof}
According to   Theorem  \ref{2th:1}, we can obtain an identity property of the fractional Riesz transform.
\begin{theorem}
The  fractional Riesz transforms satisty
$$\sum^n_{j=1}\left(R_j^{\boldsymbol{\alpha}}\right)^2=-I, \ \ \ on\  L^2(\mathbb{R}^n),$$
where $I$ is the identity operator. \end{theorem}
\begin{proof}Use the FRFT and the identity $\sum^n_{j=1}\left(-i \tilde{u}_j /|\tilde{\boldsymbol{u}}|\right)^2=-1$ to obtain
\begin{align*}
\mathcal{F}_{\boldsymbol{\alpha}}\left(\sum^n_{j=1}\left(R_j^{\boldsymbol{\alpha}}\right)^2 f\right)(\boldsymbol{u})=&\ \sum^n_{j=1}\left(-i\frac{\tilde{\boldsymbol{u}}_j}{|\tilde{\boldsymbol{u}}|}\right)^2\mathcal{F}_{\boldsymbol{\alpha}} (f)(\boldsymbol{u})\\
=&\ -\mathcal{F}_{\boldsymbol{\alpha}} (f)(\boldsymbol{u}),
\end{align*}
for any $f$ in $L^2(\mathbb{R}^n)$.
\end{proof}

\subsection{The boundedness of    singular integral  operators with chirp functions on rotation-invariant spaces}

\ \ \ \  Just like the Riesz transforms, the
 fractional Riesz transforms are singular integral operators. They are 
 special cases of more general singular integral operators whose kernels $K$ are equipped with   chirp functions 
  $$T_{\boldsymbol{\alpha}}(f)(\boldsymbol{x})={\rm p}.{\rm v}.\int_{\mathbb{R}^n}e_{-\boldsymbol{\alpha}}(\boldsymbol{x})K(\boldsymbol{x},\boldsymbol{y})e_{ \boldsymbol{\alpha}}(\boldsymbol{y})f(\boldsymbol{y})d\boldsymbol{y}={\rm p}.{\rm v}.\int_{\mathbb{R}^n}K^{\boldsymbol{\alpha}}(\boldsymbol{x},\boldsymbol{y})f(\boldsymbol{y})d\boldsymbol{y}.$$ When $\boldsymbol{\alpha}=(\frac{\pi}{2}+k_1\pi,\frac{\pi}{2}+k_2\pi,\dots, \frac{\pi}{2}+k_n\pi), k_j\in\mathbb{Z} $, $T_{\boldsymbol{\alpha}}$ can be regarded as the   classical singular integral operators: $$T(f)(\boldsymbol{x})={\rm p}.{\rm v}.\int_{\mathbb{R}^n}K(\boldsymbol{x},\boldsymbol{y})f(\boldsymbol{y})d\boldsymbol{y}.$$
Then,  we   consider the boundedness of      $T_{\boldsymbol{\alpha}}$ on rotation  invariant Banach spaces.

\begin{definition}
Suppose that $(X,\| \cdot \|_X)$  is a Banach space. We call $X$ a rotation-invariant space   if  $$\|e_{\boldsymbol{\alpha}} f\|_X=\|  f\|_X,$$
for any $f\in X$, where $ \boldsymbol{\alpha} =(\alpha_1, \dots, \alpha_n) \in \mathbb{R}^n$ with $\alpha_k\notin \pi \mathbb{Z}$ for all $ k=1,2,\dots,n$.
\end{definition}

When $K$ satisfies suitable conditions, the boundedness of $T_{\boldsymbol{\alpha}}$ and $T$ in rotation invariant space is equivalent.

\begin{theorem}
If  $X$ is a  rotation invariant space, then $T$ is bounded from $X$ to $X$ if and only if $T_{\boldsymbol{\alpha}} $ is bounded from $X$ to $X$.
\end{theorem}

\begin{proof} Let $f\in X$ and $\| T \|_{X\rightarrow X}<\infty$. We have that
\begin{align*}
  \| T_{\boldsymbol{\alpha}}(f) \|_{X} = \| T(e_{\boldsymbol{\alpha}} f) \|_{X}
\leq C\|  e_{\boldsymbol{\alpha}} f  \|_{X}
= C\|   f  \|_{X}.
\end{align*}
 Conversely, for  $\|  T_{\boldsymbol{\alpha}} \|_{X\rightarrow X}<\infty$, we obtain
\begin{align*}
  \| T(  f) \|_{X}= \| e_{-\boldsymbol{\alpha}}T(  f) \|_{X}
= \|T_{\boldsymbol{\alpha}}( e_{ -\boldsymbol{\alpha}} f) \|_{X}
\leq C\|  e_{ -\boldsymbol{\alpha}} f  \|_{X}
= C\|    f  \|_{X}.
\end{align*}
Hence, the theorem follows.
\end{proof}

\section{Chirp Hardy spaces }

\label{sect:invs}

\ \ \ \ In this section, we naturally  consider the boundedness of a singular integral operator with a chirp function on non-rotation invariant space, such as Hardy spaces. Hardy spaces are spaces of distributions  which become more singular as $p$ decreases and can be regarded as a substitute for $L^p$ when $p<1$.

 \subsection{Chirp Hardy spaces and chirp BMO spaces}

\ \ \ \ We  recall the  real variable characterization and atom  characterization of Hardy spaces.

\begin{definition}\emph{(\cite{g1})}
Let $f$ be a bounded tempered distribution on $\mathbb{R}^n$ and let $0<p<\infty$. We say that $f$ lies in the Hardy spaces $H^p (\mathbb{R}^n)$ if the Poisson maximal  function
$$M(f;P)(\boldsymbol{x})=\sup_{t>0}|(P_t \ast f)(\boldsymbol{x})|$$
lies in $L^p(\mathbb{R}^n)$. If this is the case, we set $$\|f\|_{H^p}=\|M(f;P)\|_{L^p}.$$
 \end{definition}

Before introducing the atomic  characterization of Hardy spaces, we   recall the definition of atoms.

 \begin{definition}\emph{(\cite{c,l})}
If $0<p\leq1\leq q\leq\infty, p<q,s\in \mathbb{Z}$ and $s \geq [n(\frac{1}{p}-1)]$, then $(p,q,s)$ satisfying the above conditions are said to be admissible triples, where $[\cdot]$ represents 
the greatest integer function. If the real valued function $a$ satisfies the following conditions:\\

\noindent (1)  $a\in L^q(\mathbb{R}^n)$ and supp$(a)\subset Q$, where $Q$ is a cube centered on $x_0$;\\

\noindent (2)  $\|a\|_{L^q}\leq|Q|^{\frac{1}{q}-\frac{1}{p}}$;\\

\noindent (3)  $ \displaystyle\int_{\mathbb{R}^n}a(\boldsymbol{x}){\boldsymbol{x}}^{\boldsymbol{\alpha}} d\boldsymbol{x}=0$, for any $|\boldsymbol{\alpha}|\leq s$;\\

\noindent
then $a$  is called the $(p,q,s)-$atom centered in $x_0$.
\end{definition}

 \begin{definition}\emph{(\cite{c,l})}
Suppose that $(p,q,s)$ are admissible triples. The atomic Hardy space $H^{p,q,s}_{atom}$ is  
\begin{align*}
H^{p,q,s}_{atom}(\mathbb{R}^n):=\Big\{f\in S '(\mathbb{R}^n): f(\boldsymbol{x})=\sum_j\lambda_ja_j(\boldsymbol{x}),\ {\rm where}\  a_j \ {\rm is}\   (p,q,s){\rm -}{\rm atom}, \\
 \sum_{j=1}^{\infty}|\lambda_j|^p<\infty\Big\},
\end{align*}
and the  norm in this space is defined by 
$$\|f\|_{H^{p,q,s}_{atom}}:=\inf\left(\sum_{j=1}^{\infty}\left|\lambda_j\right|^p\right)^{\frac{1}{p}}.$$
\end{definition}

 \begin{remark}
  Suppose that $f\in  S (\mathbb{R}^n)$. We have
\begin{align*}
\|f\|_{H^p} =&\ \left\|\sup_{t>0}|\left(P_t \ast f\right)|\right\|_{L^p}
 =  \left\|\sup_{t>0}\left|\int_{\mathbb{R}^n}P_t(\cdot-\boldsymbol{y})f(\boldsymbol{y})d\boldsymbol{y}\right|\right\|_{L^p},\\
\|e_{\boldsymbol{\alpha}} f\|_{H^p} =&\ \left\|\sup_{t>0}|(P_t \ast (e_{\boldsymbol{\alpha}} f))|\right\|_{L^p}
 =\left\|\sup_{t>0}\left|\int_{\mathbb{R}^n}P_t(\cdot-\boldsymbol{y})e_{\boldsymbol{\alpha}}(\boldsymbol{y})f(\boldsymbol{y})d\boldsymbol{y}\right|\right\|_{L^p}.
\end{align*}
We can clearly see that $\|e_{\boldsymbol{\alpha}} f\|_{H^p}$ depends on $\boldsymbol{\alpha}$, that is,
$$
\|f\|_{H^p}\neq \|e_{\boldsymbol{\alpha}} f\|_{H^p}.
$$ 
Note that $H^p (\mathbb{R}^n)$ is not a  rotation-invariant space.
 \end{remark}

Now let us consider the boundedness of singular integral operators with chirp functions in Hardy space when   kernel $K$ satisfies certain size and smoothness conditions. Let us  recall the definition of the $\delta$-Calder\'{o}n-Zygmund operator.

\begin{definition}\label{dfT}\emph{(\cite{l1})}
Let $T$ be a bounded linear operator.  We say that $T$ is a $\delta$-Calder\'{o}n-Zygmund operator if  $T$ is bounded on $L^2(\mathbb{R}^n)$
and  $K$ is a continuous function on  $\mathbb{R}^n\times \mathbb{R}^n\setminus\{(\boldsymbol{x},\boldsymbol{y}):\boldsymbol{x}\neq \boldsymbol{y}\}$ that satisfies \\

\noindent (1) $|K(\boldsymbol{x},\boldsymbol{y})|\leq\frac{C}{|\boldsymbol{x}-\boldsymbol{y}|^n},\ \boldsymbol{x}\neq \boldsymbol{y};$ \\

\noindent (2) $|K(\boldsymbol{x},\boldsymbol{y})-K(\boldsymbol{x},\boldsymbol{z})|+|K(\boldsymbol{y},\boldsymbol{x})-K(\boldsymbol{z},\boldsymbol{x})|\leq C\frac{|\boldsymbol{y}-\boldsymbol{z}|^\delta}{|\boldsymbol{x}-\boldsymbol{z}|^{n+\delta}},\ \ \ {\rm if}\  |\boldsymbol{x}-\boldsymbol{z}|>2|\boldsymbol{y}-\boldsymbol{z}|,\ 0<\delta\leq1 ; $\\

\noindent (3) For $f,g\in S(R^n) $  and supp $f$ $\cap$ supp $g=\varnothing$, one has
$$(Tf,g)=\int K(\boldsymbol{x},\boldsymbol{y})f(\boldsymbol{y})g(\boldsymbol{x})d\boldsymbol{y}d\boldsymbol{x}.$$

 \end{definition}

 \begin{lemma}\emph{(\cite{l1})}
Suppose that $T$ is a $\delta$-Calder\'{o}n-Zygmund operator and its conjugate operator $T^\ast=0$. Then, $T$ can be extended to the  bounded operator from $ H^p(\mathbb{R}^n)  $ to $ H^p (\mathbb{R}^n) $, where $0<\delta\leq1$ and $\frac{n}{n+\delta}<p\leq1$.
\end{lemma}
By standard calculations, we have the following estimates for   $K^{\boldsymbol{\alpha}}$:
 \begin{align*}
 |K^{\boldsymbol{\alpha}}(\boldsymbol{x},\boldsymbol{y})-K^{\boldsymbol{\alpha}}(\boldsymbol{x},\boldsymbol{z})|\leq&\  |K (\boldsymbol{x},\boldsymbol{y})-K (\boldsymbol{x},\boldsymbol{z})|+L_{\boldsymbol{\alpha}}{(\boldsymbol{y},\boldsymbol{z})}|K (\boldsymbol{x},\boldsymbol{z})||\boldsymbol{y}-\boldsymbol{z}|,
 \end{align*}
where 
$$
L_{\boldsymbol{\alpha}}{(\boldsymbol{y},\boldsymbol{z})}=\left|\triangledown e_{{\boldsymbol{\alpha}}}({\boldsymbol{w}})\right|=\sqrt{\sum^n_{k=1}|e_{\boldsymbol{\alpha}}({\boldsymbol{w}}) \cot\alpha_k w_k|^2} ,
$$
and $\boldsymbol{w}=\boldsymbol{z}+(\theta_1(y_1-z_1),\dots,\theta_n(y_n-z_n)))$ for $\theta_j\in(0,1)$.
 It is known that for  $\boldsymbol{\alpha}=(\frac{\pi}{2}+k_1\pi,\frac{\pi}{2}+k_2\pi,\dots, \frac{\pi}{2}+k_n\pi)$, $K^{\boldsymbol{\alpha}}$ satisfies the $\delta$-Calder\'{o}n-Zygmund operator   kernel condition $(2)$.
 It is obvious that $K^{\boldsymbol{\alpha}}$ satisfies $(1)$ in Definition \ref{dfT}, but $(2)$  in Definition \ref{dfT} is not guaranteed.

We now define a new class of Hardy space with chirp functions.
 \begin{definition}
Let $f$ be a bounded tempered distribution on $\mathbb{R}^n$ and let $0<p<\infty$. We say that $f$ lies in the chirp Hardy space $H_{\boldsymbol{\alpha}}^p (\mathbb{R}^n)$ for $ \boldsymbol{\alpha } 
=(\alpha_1, \dots , \alpha_n) \in \mathbb{R}^n$  with $\alpha_k\notin \pi \mathbb{Z}$ for all  $k=1,2,\dots,n$, if the  Poisson maximal  function with chirp function
$$M_{\boldsymbol{\alpha}}(f;P) =\sup_{t>0}|(P_t \ast (e_{\boldsymbol{\alpha}} f)) |$$
lies in $L^p(\mathbb{R}^n)$. If this is the case, we set $$\|f\|_{H^p_{\boldsymbol{\alpha}}}=\|M_{\boldsymbol{\alpha}}(f;P)\|_{L^p}.$$
\end{definition}

 \begin{lemma}\label{3th:1}\emph{(\cite{g2})}
The Hardy space $H^p (\mathbb{R}^n)$ is a complete space.
\end{lemma}

 \begin{theorem}\label{3th:2}
The chirp Hardy space $H_{\boldsymbol{\alpha}}^p (\mathbb{R}^n)$ is a complete space.
\end{theorem}

\begin{proof}
Let $\{f_k\}$ be  a Cauchy sequence in $H_{\boldsymbol{\alpha}}^p (\mathbb{R}^n)$. Then, $\{e_{\boldsymbol{\alpha}}{f_k}\}$ is a Cauchy sequence in $H^p (\mathbb{R}^n)$.
 By Lemma \ref{3th:1},   there exists an $\bar{f}\in H^p (\mathbb{R}^n)$  such that
$$\lim_{k\rightarrow\infty}\|e_{\boldsymbol{\alpha}}{f_k}-\bar{f}\|_{H^p}=0.$$
The above  identity is  rewritten as
$$\lim_{k\rightarrow\infty}\|{f_k}-e_{-\boldsymbol{\alpha}}\bar{f}\|_{H_{\boldsymbol{\alpha}}^p}=0.$$
Since $\bar{f}\in H^p (\mathbb{R}^n)$, we obtain $f:=e_{-\boldsymbol{\alpha}}\bar{f}\in H_{\boldsymbol{\alpha}}^p (\mathbb{R}^n)$,  which completes the proof of the theorem.
\end{proof}

It is  known that the dual space of $H^1$ is the $BMO$ space. To study the dual space of the chirp Hardy space,   we define a new $BMO$ space with a chirp function as follows.
 \begin{definition}
Suppose that  $f$ is a   locally integrable function on $\mathbb{R}^n$. Define the chirp $BMO$ space as $$BMO^{\boldsymbol{\alpha}}(\mathbb{R}^n)=\{f:\|f\|_{{BMO}^{\boldsymbol{\alpha}}}<\infty\}.$$
Let
$$\|f\|_{{BMO}^{\boldsymbol{\alpha}}}=\sup_{Q}\frac{1}{|Q|}\int_{Q}|e_{\boldsymbol{\alpha}}(\boldsymbol{x})f(\boldsymbol{x})-{\rm Avg}_{Q}(e_{\boldsymbol{\alpha} }f)|d\boldsymbol{x},$$
where the supremum is taken over all cubes $Q$ in $\mathbb{R}^n$ and $ \boldsymbol{\alpha} 
=(\alpha_1, \dots , \alpha_n) \in \mathbb{R}^n$ with $\alpha_k\notin \pi \mathbb{Z}, \ k=1,2,\dots,n$.

\end{definition}

 \begin{lemma}\label{3th:9}\emph{(\cite{d,g2})}
 $BMO$    is a complete space.
\end{lemma}

 \begin{theorem}
 $BMO^{\boldsymbol{\alpha}}$ is   complete.
\end{theorem}

The proof follows the same pattern as that of
  Theorem \ref{3th:2} and is based on  Lemma \ref{3th:9}.

 \subsection{Dual spaces of   chirp Hardy spaces}

  \ \ \ \  Let $\boldsymbol{\alpha}=(\alpha_1,\alpha_2,\dots,\alpha_n)\in \mathbb{R}^n$ with $\alpha_k\notin \pi \mathbb{Z}$ for all  $k=1,2,\dots,n$.
  We discuss the dual spaces of chirp Hardy spaces $H^p_{\boldsymbol{\alpha}}(\mathbb{R}^n)  $ for $0<p\leq1$.  When $p=1$, we have the following theorem.
 \begin{theorem}\label{3th:3}
 $(H^1_{\boldsymbol{\alpha}})^*(\mathbb{R}^n)=BMO^{-\boldsymbol{\alpha}}(\mathbb{R}^n)$.  That is,\\
(1) For any $g\in BMO^{-\boldsymbol{\alpha}}(\mathbb{R}^n)$,
$$L(f):=\int_{\mathbb{R}^n}f(x)g(x)dx$$
is  a bounded linear functional on $H^1_{\boldsymbol{\alpha}}(\mathbb{R}^n)$ and $\|L\|\leq C\|g\|_{BMO^{-\boldsymbol{\alpha}}}$.\\
(2) For any bounded linear functionals $L$ defined on $H^1_{\boldsymbol{\alpha}}(\mathbb{R}^n)$, there exists a $g\in BMO^{-\boldsymbol{\alpha}}(\mathbb{R}^n)$ such that
$$L(f):=\int_{\mathbb{R}^n}f(\boldsymbol{x})g(\boldsymbol{x})d\boldsymbol{x}, $$
 for   any $f\in H^1_{\boldsymbol{\alpha}}(\mathbb{R}^n)$, $\|g\|_{BMO^{-{\boldsymbol{\alpha}}}}\leq C\|L\|$.
\end{theorem}

Before proving the theorem, we   need to recall    some known results.
\begin{lemma}\emph{(\cite{l})}\label{3th:4}
For all $1\leq q\leq\infty$, it   follows that $H^{p,q,s}_{atom}(\mathbb{R}^n)=H^p(\mathbb{R}^n)$ and $\|f\|_{H^{p,q,s}_{atom}}\approx \|f\|_{H^ p}$ for $f\in H^p(\mathbb{R}^n)$.
\end{lemma}
\begin{lemma}\label{3th:5}\emph{(\cite{d}, \cite{f}, \cite{fs}) } $(H^1)^*(\mathbb{R}^n)=BMO(\mathbb{R}^n)$.
\end{lemma}
Now, we will go back to prove   Theorem \ref{3th:3}.
\begin{proof}
For any $f\in H^1_{\boldsymbol{\alpha}}(\mathbb{R}^n)$, by Lemma \ref{3th:4}, we  obtain   $e_{\boldsymbol{\alpha}} f\in H^1(\mathbb{R}^n)$ such  that $e_{\boldsymbol{\alpha}} f=\sum^\infty_{j=1}\lambda_ja_j$, where $a_j$ is $(1,\infty,0)$-atom. Suppose that $g\in BMO^{-\boldsymbol{\alpha}}(\mathbb{R}^n)$. We   have
\begin{align*}
|L(f)|=& \ \left|\int_{\mathbb{R}^n}f(\boldsymbol{x})g(\boldsymbol{x})d\boldsymbol{x}\right|\\
=&\  \left|\int_{\mathbb{R}^n} e_{-\boldsymbol{\alpha}}(\boldsymbol{x}) \sum^\infty_{j=1}\lambda_ja_j(\boldsymbol{x})g(\boldsymbol{x})d\boldsymbol{x}\right|\\
=&\  \left| \sum^\infty_{j=1}\lambda_j \int_{Q_j }a_j(\boldsymbol{x})[ e_{-\boldsymbol{\alpha}}(\boldsymbol{x})g(\boldsymbol{x})-(e_{-\boldsymbol{\alpha}} g)_Q ]d\boldsymbol{x}\right|\\
\leq&\ \sum^\infty_{j=1}|\lambda_j |\frac{1}{|Q_j|}\int_{Q_j }  | e_{-\boldsymbol{\alpha}}(\boldsymbol{x})g(\boldsymbol{x})-(e_{-\boldsymbol{\alpha}} g)_Q |d\boldsymbol{x}\\
\leq&\  C\sum^\infty_{j=1}|\lambda_j|  \|g\|_{BMO^{-\boldsymbol{\alpha}}}\\
=&\  C\|e_\alpha f\|_{H^1}  \|g\|_{BMO^{-\boldsymbol{\alpha}}}\\
=&\  C\|  f\|_{H^1_\alpha}  \|g\|_{BMO^{-\boldsymbol{\alpha}}}.
\end{align*}
Then $L$ is  a bounded linear functional on $H^1_{\boldsymbol{\alpha}}(\mathbb{R}^n)$ and $\|L\|\leq C\|g\|_{BMO^{-{\boldsymbol{\alpha}}}}$.

Now given an $L\in(H^1_{\boldsymbol{\alpha}})^*(\mathbb{R}^n)$, denote $\tilde{L}(f):=L(e_{-\boldsymbol{\alpha}}f)$ for any $f\in H^1(\mathbb{R}^n)$. We have
$$|\tilde{L}(f)|=|L(e_{-\boldsymbol{\alpha}}f)|\leq \|L\|\  \|  e_{-\boldsymbol{\alpha}}f\|_{H^1_{\boldsymbol{\alpha}}}=\|L\| \  \|   f\|_{H^1 }.$$
Hence,  $\tilde{L}\in(H^1)^*(\mathbb{R}^n)$ and $\|\tilde{L}\| \leq\|L\| $.

From  Lemma \ref{3th:5},    $\tilde{g}\in BMO(\mathbb{R}^n) $ exists such that $\tilde{L}(f)=\int_{\mathbb{R}^n}f(\boldsymbol{x})\tilde{g}(\boldsymbol{x})d\boldsymbol{x}$ for any $f\in H^1(\mathbb{R}^n)$ and $\|\tilde{g}\|_{BMO}\leq C\|\tilde{L}\|$.

For any $f\in H^1_\alpha(\mathbb{R}^n)$,  $e_{\boldsymbol{\alpha}}f\in H^1(\mathbb{R}^n)$. We  obtain
 \begin{align*}
 L(f)=\tilde{L}(e_{\boldsymbol{\alpha}} f)
 =\int_{\mathbb{R}^n} e_{\boldsymbol{\alpha}}(\boldsymbol{x})f(\boldsymbol{x})\tilde{g}(\boldsymbol{x})d\boldsymbol{x}
 =\int_{\mathbb{R}^n}f(\boldsymbol{x}) g (\boldsymbol{x})d\boldsymbol{x},
 \end{align*}
 where $ g= e_{\boldsymbol{\alpha}}\tilde{g}$. Since $\tilde{g}\in BMO(\mathbb{R}^n)$ and $\|\tilde{g}\|_{BMO}\leq C\|\tilde{L}\|$, $ g \in BMO^{-\boldsymbol{\alpha}}(\mathbb{R}^n)$ and $\| g \|_{BMO^{-\boldsymbol{\alpha}}}=\|\tilde{g}\|_{BMO}\leq C\|\tilde{L}\|\leq C\|L\|,$  which completes the proof of the theorem.
 \end{proof}
 \begin{remark}
When $\boldsymbol{\alpha}=(\frac{\pi}{2}+k_1\pi,\frac{\pi}{2}+k_2\pi,\dots, \frac{\pi}{2}+k_n\pi), k_j\in \mathbb{Z} $ for $j=1,2,\dots,m$, $(H^1_{\boldsymbol{\alpha}})^*(\mathbb{R}^n)=BMO^{-\boldsymbol{\alpha}}(\mathbb{R}^n)$  reduces to $(H^1)^*(\mathbb{R}^n)=BMO(\mathbb{R}^n)$.
\end{remark}
Now let us  proceed to consider  the dual space of the chirp  Hardy space when $0<p<1$.
\begin{lemma} \label{3th:6}\emph{(\cite{tw})} For $g\in L^1_{loc}(\mathbb{R}^n)$, $Q$   is an any cube in $\mathbb{R}^n$ and $s\in\mathbb{Z}^+$. Then there exists a unique polynomial $P_Q(g)$ whose  degree does not exceed $s$ that satisfies
$$\int_Q[g(\boldsymbol{x})-P_Q(g)(\boldsymbol{x})]\boldsymbol{x}^{\alpha}d\boldsymbol{x}=0,\ \ \ 0\leq|\alpha|\leq s.$$
\end{lemma}
 \begin{definition} \emph{(\cite{tw})}
For $s\in\mathbb{Z}^+$, $0\leq[n\beta]\leq s$ and $1\leq q'\leq\infty$, the Campanato-Meyers space $L(\beta,q',s)(\mathbb{R}^n)$ is defined as the set of   locally integrable functions $g$  that satisfy
$$\|g\|_{L(\beta,q',s)}=\sup_{Q\subset\mathbb{R}^n}|Q|^{-\beta}\left[\int_Q|g(\boldsymbol{x}) -P_Q(g)(\boldsymbol{x}) |^{q'}\frac{d\boldsymbol{x}}{|Q|}\right]^{\frac{1}{q'}}<\infty,$$
where $P_Q(g)$  is determined by Lemma \ref{3th:6}.
\end{definition}
\begin{lemma}\label{3th:7}\emph{(\cite{cw,w1})}
$(H^p)^*(\mathbb{R}^n)=L(\frac{1}{p},q',s)(\mathbb{R}^n)$, where $0<p<1\leq q\leq\infty$, $s\in\mathbb{Z}$, $s\geq n(\frac{1}{p}-1)$ and $1/q+1/q'=1$.
\end{lemma}
Now let us  define a new Campanato-Meyers space with chirps as follows:
 \begin{definition}
For $s\in\mathbb{Z}^+$, $0\leq[n\beta]\leq s$ and $1\leq q'\leq\infty$. The chirp Campanato-Meyers space $L_{\boldsymbol{\alpha}}(\beta,q',s)(\mathbb{R}^n)$ is defined as the set of   locally integrable functions $g$  that satisfy
$$\|g\|_{L_{\boldsymbol{\alpha}}(\beta,q',s)}=\sup_{Q\subset\mathbb{R}^n}|Q|^{-\beta}\left[\int_Q|e_{\boldsymbol{\alpha}}(\boldsymbol{x})g (\boldsymbol{x}) -P_Q(e_{\boldsymbol{\alpha}}g)(\boldsymbol{x}) |^{q'}\frac{d\boldsymbol{x}}{|Q|}\right]^{\frac{1}{q'}}<\infty,$$
where $P_Q(e_{\boldsymbol{\alpha}}g)$  is determined by Lemma \ref{3th:6}.
\end{definition}
\begin{theorem}
$(H_{\boldsymbol{\alpha}}^p)^*(\mathbb{R}^n)=L_{-\boldsymbol{\alpha}}(\frac{1}{p},q',s)(\mathbb{R}^n)$, where $0<p<1\leq q\leq\infty$, $s\in\mathbb{Z}$, $s\geq n(\frac{1}{p}-1)$ and $1/q+1/q'=1$.
\end{theorem}
\begin{proof}
For any $f\in H^p_{\boldsymbol{\alpha}}(\mathbb{R}^n)$, by Lemma \ref{3th:4}, we have  $e_{\boldsymbol{\alpha}} f\in H^p(\mathbb{R}^n)$ such  that $e_{\boldsymbol{\alpha}} f=\sum^\infty_{j=1}\lambda_ja_j$, where $a_j$ is a $(p,q,s)$-atom. Suppose that $g\in L_{-\boldsymbol{\alpha}}(\frac{1}{p},q',s)(\mathbb{R}^n)$. We have 
\begin{align*}
|L(f)|=&\  \left|\int_{\mathbb{R}^n} e_{-\boldsymbol{\alpha}}(\boldsymbol{x}) \sum^\infty_{j=1}\lambda_ja_j(\boldsymbol{x})g(\boldsymbol{x})d\boldsymbol{x}\right|\\
=&\  \left| \sum^\infty_{j=1}\lambda_j \int_{Q_j }a_j(\boldsymbol{x})[ e_{-\alpha}(\boldsymbol{x})g(\boldsymbol{x})-P_{Q_j}(e_{-\boldsymbol{\alpha}}(\boldsymbol{x})g(\boldsymbol{x})) ]d\boldsymbol{x}\right|\\
\leq&\ \sum^\infty_{j=1}|\lambda_j |\|a_j\|_{L^q}\left(\int_{Q_j }  | e_{-\boldsymbol{\alpha}}(\boldsymbol{x})g(\boldsymbol{x})-P_{Q_j}(e_{-\boldsymbol{\alpha}}(\boldsymbol{x})g(\boldsymbol{x}) ) |^{q'}d\boldsymbol{x}\right)^{\frac{1}{q'}}\\
\leq&\  C\left(\sum^\infty_{j=1}|\lambda_j|^{p}\right)^{\frac{1}{p}}  \|g\|_{L_{-\boldsymbol{\alpha}}(\frac{1}{p},q',s)}\\
=&\  C\|  f\|_{H^p_\alpha}  \|g\|_{L_{-\boldsymbol{\alpha}}(\frac{1}{p},q',s)}.
\end{align*}
Then $L$ is  a bounded linear functional on $H^p_{\boldsymbol{\alpha}}(\mathbb{R}^n)$ and $\|L\|\leq C\|g\|_{L_{-\boldsymbol{\alpha}}(\frac{1}{p},q',s)}$.

Given  $L\in(H^p_{\boldsymbol{\alpha}})^*(\mathbb{R}^n)$, denote $\tilde{L}(f):=L(e_{-\boldsymbol{\alpha}}f)$ for any $f\in H^p(\mathbb{R}^n)$. We have
$$|\tilde{L}(f)|=|L(e_{-\boldsymbol{\alpha}}f)|\leq \|L\|\  \|  e_{-\boldsymbol{\alpha}}f\|_{H^p_{\boldsymbol{\alpha}}}=\|L\| \  \|   f\|_{H^p }.$$
Hence,  $\tilde{L}\in(H^p)^*(\mathbb{R}^n)$ and $\|\tilde{L}\| \leq\|L\| $.

By Lemma \ref{3th:7},  there exists $\tilde{g}\in L(\frac{1}{p},q',s)(\mathbb{R}^n)$ such that $\tilde{L}(f)=\int_{\mathbb{R}^n}f(\boldsymbol{x})\tilde{g}(\boldsymbol{x})d\boldsymbol{x}$ for any $f\in H^p(\mathbb{R}^n)$ and $\|\tilde{g}\|_{L_{\boldsymbol{\alpha}}(\frac{1}{p},q',s)}\leq C\|\tilde{L}\|$.

For any $f\in H^p_\alpha(\mathbb{R}^n)$,  $e_{\boldsymbol{\alpha}}f\in H^p(\mathbb{R}^n)$. We  obtain
 \begin{align*}
 L(f)=&\ \tilde{L}(e_{\boldsymbol{\alpha}} f)
 =  \int_{\mathbb{R}^n} e_{\boldsymbol{\alpha}}(\boldsymbol{x})f(\boldsymbol{x})\tilde{g}(\boldsymbol{x})d\boldsymbol{x}
 = \int_{\mathbb{R}^n}f(\boldsymbol{x}) g (\boldsymbol{x})d\boldsymbol{x},
 \end{align*}
 where $ g= e_{\boldsymbol{\alpha}}\tilde{g}$. Since $\tilde{g}\in L(\frac{1}{p},q',s)(\mathbb{R}^n)$ and $\|\tilde{g}\|_{L(\frac{1}{p},q',s)}\leq C\|\tilde{L}\|$, then $ g \in L_{\boldsymbol{-\alpha}}(\frac{1}{p},q',s)$ and $\| g \|_{L_{\boldsymbol{-\alpha}}(\frac{1}{p},q',s)}=\|\tilde{g}\|_{L(\frac{1}{p},q',s)}\leq C\|\tilde{L}\|\leq C\|L\|,$  which completes the proof of the theorem.
 \end{proof}

 \begin{remark}
When $0<p<1$ and $\boldsymbol{\alpha}=(\frac{\pi}{2}+k_1\pi,\frac{\pi}{2}+k_2\pi,\dots, \frac{\pi}{2}+k_n\pi), k_j\in \mathbb{Z} $ for $j=1,2,\dots,m$, $(H_{\boldsymbol{\alpha}}^p)^*(\mathbb{R}^n)=L_{-\boldsymbol{\alpha}}(\frac{1}{p},q',s)(\mathbb{R}^n)$    reduces to $(H^p)^*(\mathbb{R}^n)=L(\frac{1}{p},q',s)(\mathbb{R}^n)$.
\end{remark}

 \subsection{Characterization of the boundedness of singular integral operators with chirp functions on chirp Hardy spaces}

\ \ \ \  In this subsection we obtain  a characterization of the boundedness of $T_{\boldsymbol{\alpha}}$ in $H^p_{\boldsymbol{\alpha}}$.
\begin{theorem}
$T_{\boldsymbol{\alpha}}$ is bounded from $H^p_{\boldsymbol{\alpha}}(\mathbb{R}^n)$ to $H^p_{\boldsymbol{\alpha}}(\mathbb{R}^n)$ if and only if $T$ is bounded from $H^p (\mathbb{R}^n)$ to $H^p (\mathbb{R}^n)$, where $\boldsymbol{\alpha}=(\alpha_1,\alpha_2,\dots,\alpha_n)\in \mathbb{R}^n$ with $\alpha_k\notin \pi \mathbb{Z}$.
\end{theorem}
\begin{proof}
Suppose that $f\in S'$ and $\|T_{\boldsymbol{\alpha}}\|_{H^p_{\boldsymbol{\alpha}}\rightarrow H^p_{\boldsymbol{\alpha}}}<\infty$. Then, we have
\begin{align*}
\| T(f) \|_{H^p }=&\ \Bigg \| \sup_{t>0} |P_t \ast (e_{\boldsymbol{\alpha}}(e_{-\boldsymbol{\alpha}}Tf))|\Bigg\|_{L^p}\\
=&\  \|T_{\boldsymbol{\alpha}} (e_{-\boldsymbol{\alpha}}f) \|_{H^p_{\boldsymbol{\alpha}}}\\
\leq&\ C\|  e_{-\boldsymbol{\alpha}} f  \|_{H^p_{\boldsymbol{\alpha}}}\\
=&\ C\|   f   \|_{H^p }.
\end{align*}
 Conversely, when $\|T\|_{H^p\rightarrow H^p}<\infty$, we obtain
\begin{align*}
\| T_{\boldsymbol{\alpha}}(f) \|_{H^p_{\boldsymbol{\alpha}}}  = \Bigg\| \sup_{t>0} |P_t \ast ( T( e_{\boldsymbol{\alpha}} f))|\Bigg\|_{L^p}
=  \|T ( e_{ \boldsymbol{\alpha}} f) \|_{H^p }
 \leq C\|  e_{ \boldsymbol{\alpha}} f  \|_{H^p }
 = C\|    f  \|_{H^p_{\boldsymbol{\alpha}} }.
\end{align*}
Hence, the theorem follows.
\end{proof}

\section{Application of the fractional Riesz transform in partial differential equations}

\label{sect:Lp}

The fractional Riesz transforms can be used   to reconcile   various combinations of partial derivatives of functions.  We first established the derivative formula of the FRFT.
\begin{lemma}\emph{(FRFT derivative formula)}\label{4th:1}
Suppose that $f\in L^1(R^n)$. If $e_{\boldsymbol{ \alpha}}f$ is absolutely continuous on $\mathbb{R}^n$ with respect to the $k$th  variable, we have
$$\mathcal{F}_{\boldsymbol{\alpha}}\left(e_{-\boldsymbol{\alpha}}(\boldsymbol{y})\frac{\partial[e_{ \boldsymbol{\alpha}}(\boldsymbol{y})f(\boldsymbol{y})]}{\partial y_k}\right)(\boldsymbol{x})=ix_k\csc \alpha_k\mathcal{F}_\alpha (f)(\boldsymbol{x}),$$
where $\boldsymbol{x}=(x_1,x_2,\dots,x_n)$ and $\boldsymbol{\alpha}=(\alpha_1,\alpha_2,\dots,\alpha_n)\in \mathbb{R}^n$ with $\alpha_k\notin \pi \mathbb{Z}, \ k=1,2,\dots,n$.\end{lemma}
\begin{proof}Since $e_{\boldsymbol{ \alpha}}f$ is absolutely continuous on $\mathbb{R}^n$ with respect to the $k$th  variable, we can get that $\frac{\partial[e_{ \boldsymbol{\alpha}}(\boldsymbol{y})f(\boldsymbol{y})]}{\partial y_k}\in L^1(\mathbb{R}^n)$. For  $f\in L^1(\mathbb{R}^n)$, we have
\begin{align*}
\mathcal{F}_{\boldsymbol{\alpha}}\left(e_{-\boldsymbol{\alpha}}(\boldsymbol{y})\frac{\partial[e_{ \boldsymbol{\alpha}}(\boldsymbol{y})f(\boldsymbol{y})]}{\partial y_k}\right)(\boldsymbol{x}) =&\ \frac{c(\boldsymbol{\alpha})}{\left(\sqrt{2\pi}\right)^{n}}\int_{\mathbb{R}^n}\left(e_{-\boldsymbol{\alpha}}(\boldsymbol{y})\frac{\partial[e_{ \boldsymbol{\alpha}}(\boldsymbol{y})f(\boldsymbol{y})]}{\partial y_k}\right)e_{ \alpha}(\boldsymbol{y})\\
 &\ \ \ \ \ \ \ \ \ \ \ \ \ \ \ \ \times e_{ \alpha}(\boldsymbol{x})e^{-i\sum^n_{j=1}x_jy_j\csc\alpha_j}d\boldsymbol{y}\\
 =&\  \frac{c({\boldsymbol{\alpha}})}{\left(\sqrt{2\pi}\right)^{n}}e_{ \boldsymbol{\alpha}}(\boldsymbol{x})\int_{\mathbb{R}^n} \frac{\partial[e_{ \boldsymbol{\alpha}}(\boldsymbol{y})f(\boldsymbol{y})]}{\partial y_k} e^{-i\sum^n_{j=1}x_jy_j\csc\alpha_j}d\boldsymbol{y}\\
 =&\ \frac{c({\boldsymbol{\alpha}})e_{ \boldsymbol{\alpha}}(\boldsymbol{x})}{\left(\sqrt{2\pi}\right)^{n}}\int_{\mathbb{R}^{n-1}}\int_{\mathbb{R} }  \frac{\partial[e_{ \alpha}(\boldsymbol{y})f(\boldsymbol{y})]}{\partial y_k}e^{-ix_ky_k\csc \alpha_k } dy_k \\
 &\ \ \ \ \ \ \ \ \ \ \ \ \ \ \ \ \times\prod^n_{j=1,j\neq k}e^{-ix_jy_j\csc\alpha_j}\prod^n_{j=1,j\neq k}dy_j.
\end{align*}
As $e_{ \boldsymbol{\alpha}}f$  is absolutely continuous on $\mathbb{R}^n$ with respect to the $k$th variable, an integration by parts yields 
$$\int_{\mathbb{R} }  \frac{\partial[e_{ \boldsymbol{\alpha}}(\boldsymbol{y})f(\boldsymbol{y})]}{\partial y_k}e^{-ix_ky_k\csc \alpha_k}dy_k  = ix_k \csc \alpha_k \int_{\mathbb{R} }e_{ \boldsymbol{\alpha}}(\boldsymbol{y})f(\boldsymbol{y})e^{-ix_ky_k\csc \alpha_k }dy_k.$$
Then
\begin{align*}
\mathcal{F}_{{\boldsymbol{\alpha}}}\left(e_{-\boldsymbol{\alpha}}(\boldsymbol{y})\frac{\partial[e_{ \boldsymbol{\alpha}}(\boldsymbol{y})f(\boldsymbol{y})]}{\partial y_k}\right)(\boldsymbol{x}) =&\ ix_k \csc \alpha_k\frac{c(\boldsymbol{\alpha})}{(\sqrt{2\pi})^{n}}\int_{\mathbb{R}^n}e_{ \alpha}(\boldsymbol{y})f(\boldsymbol{y})e_{ \boldsymbol{\alpha}}(\boldsymbol{x})e^{-i\sum^n_{j=1}x_jy_j\csc\alpha_j}d\boldsymbol{y}\\
=&\ ix_k \csc \alpha_k\mathcal{F}_{\alpha} (f)(\boldsymbol{x}),
\end{align*}
which completes the proof of the lemma.
\end{proof}
\begin{lemma}\emph{(FRFT derivative formula)}\label{4th:2}
Suppose that $f\in L^1(\mathbb{R}^n)$ and $x_kf(\boldsymbol{x})\in L^1(\mathbb{R}^n)$. Then, we  have
$$\frac{\partial(e_{-\boldsymbol{\alpha}}(\boldsymbol{x})\mathcal{F}_{\boldsymbol{\alpha}} (f)(\boldsymbol{x}))}{\partial x_k}=e_{-\boldsymbol{\alpha}}(\boldsymbol{x})\mathcal{F}_{\boldsymbol{\alpha}}(-iy_k\csc\alpha_kf(\boldsymbol{y}))(\boldsymbol{x}),$$
for $\boldsymbol{\alpha}=(\alpha_1,\alpha_2,\dots,\alpha_n)\in \mathbb{R}^n$ with $\alpha_k\notin \pi \mathbb{Z},\ k=1,2,\dots,n.$\end{lemma}

\begin{proof} Let $\Delta_k=(0,\dots,0,\delta,0,\dots,0), \ \delta\neq0,\  and\  let \ \delta\  be\  the\  kth\  variable$. Then
\begin{align*}
\frac{\partial(e_{-\boldsymbol{\alpha}}(\boldsymbol{x})\mathcal{F}_{\boldsymbol{\alpha}} (f)(\boldsymbol{x}))}{\partial x_k} =&\ \lim_{\delta\rightarrow0}\frac{e_{-\boldsymbol{\alpha}}(\boldsymbol{x}+\triangle k)\mathcal{F}_{\boldsymbol{\alpha}} (f)(\boldsymbol{x}+\triangle k)-e_{-\boldsymbol{\alpha}}(\boldsymbol{x})\mathcal{F}_{\boldsymbol{\alpha}} (f)(\boldsymbol{x})}{\delta}\\
 =&\ \lim_{\delta\rightarrow0}\frac{1}{\delta} \Bigg( \int_{\mathbb{R}^n}\frac{c(\boldsymbol{\alpha})}{\left(\sqrt{2\pi}\right)^{n}}f(\boldsymbol{y}) e_{ \boldsymbol{\alpha}}(\boldsymbol{y})e^{-i(\boldsymbol{x}+\triangle k)\cdot\tilde{\boldsymbol{y}}}d\boldsymbol{y} \\
 & \ \ \ \ \ \ \ \ \ \ \ \ \ \ \ \ \  \ \ \ \ -\int_{\mathbb{R}^n}\frac{c(\boldsymbol{\alpha})}{\left(\sqrt{2\pi}\right)^{n}}f(\boldsymbol{y}) e_{ \boldsymbol{\alpha}}(\boldsymbol{y})e^{-i \boldsymbol{x}  \cdot\tilde{\boldsymbol{y}}}d\boldsymbol{y}\Bigg)\\
 =&\ \lim_{\delta\rightarrow0}\frac{1}{\delta} \left( \int_{\mathbb{R}^n}\frac{c(\boldsymbol{\alpha})}{\left(\sqrt{2\pi}\right)^{n}}f(\boldsymbol{y}) e_{ \boldsymbol{\alpha}}(\boldsymbol{y})e^{-i\boldsymbol{x}\cdot\tilde{\boldsymbol{y}}}(e^{-i \delta \tilde{y}_k}-1)d\boldsymbol{y} \right).
\end{align*}
By $|\frac{e^{-i \delta \tilde{y}_k}-1}{\delta}|\leq2\pi|\tilde{y}_k|$, $x_kf(\boldsymbol{x})\in L^1(\mathbb{R}^n)$ and the Lebesgue dominated convergence theorem we write 
\begin{align*}
\frac{\partial(e_{-\boldsymbol{\alpha}}(\boldsymbol{x})\mathcal{F}_{\boldsymbol{\alpha}} (f)(\boldsymbol{x}))}{\partial x_k}=& \  \int_{\mathbb{R}^n}\frac{c({\boldsymbol{\alpha}})}{\left(\sqrt{2\pi}\right)^{n}}f(\boldsymbol{y}) e_{ \boldsymbol{\alpha}}(\boldsymbol{y})e^{-i \boldsymbol{x} \cdot\tilde{\boldsymbol{y}}}\left[\lim_{\delta\rightarrow0}\frac{1}{\delta}(e^{-i \delta \tilde{y}_k}-1)\right]d\boldsymbol{y}  \\
=& \  \int_{\mathbb{R}^n}\frac{c(\boldsymbol{\alpha})}{\left(\sqrt{2\pi}\right)^{n}}f(\boldsymbol{y}) e_{ \boldsymbol{\alpha}}(\boldsymbol{y})e^{-i \boldsymbol{x} \cdot\tilde{\boldsymbol{y}}}(-i\tilde{y}_k)d\boldsymbol{y}  \\
=& \  e_{{-\boldsymbol{\alpha}}}(\boldsymbol{x})\mathcal{F}_{\boldsymbol{\alpha}}(-iy_k\csc\alpha_kf(\boldsymbol{y}))(\boldsymbol{x}),
\end{align*}
where $\tilde{\boldsymbol{y}}=(\tilde{y}_1,\dots,\tilde{y}_n)=(y_1\csc\alpha_1,\dots,y_n\csc\alpha_n).$
\end{proof}

\subsection{ Application in the a priori bound estimates  of partial differential equations}

\ \ \ \ We will next introduce the  applications of  $R_j^{\boldsymbol{\alpha}}$ in the priori bound estimates.
\begin{theorem}\label{4th:3}
Suppose that $f\in S (\mathbb{R}^2)$. Then, we have the  a priori bound
$$\Big\|\frac{\partial(e_{\boldsymbol{\alpha}} f)}{\partial y_1}\Big\|_{L^p}+\Big\|\frac{\partial(e_{\boldsymbol{\alpha}} f)}{\partial y_2}\Big\|_{L^p}\leq C\Big\|e_{-\boldsymbol{\alpha}}(\boldsymbol{y})\frac{\partial(e_{\boldsymbol{\alpha}}(\boldsymbol{y})f(\boldsymbol{y}))}{\partial y_1} +ie_{\boldsymbol{\alpha}}(\boldsymbol{y})\frac{\partial(e_{\boldsymbol{\alpha}}(\boldsymbol{y})f(\boldsymbol{y}))}{\partial y_2}\Big\|_{L^p},\ \ 1<p<\infty,$$
where $\boldsymbol{\alpha}=(\alpha_1,\alpha_2)\in \mathbb{R}^2$ with $\alpha_k\notin \pi \mathbb{Z},\ k=1,2.$
\end{theorem}
To prove    Theorem \ref{4th:3}, we need the following lemma.
\begin{lemma}\label{4th:5}
Let $f\in S (\mathbb{R}^2)$. We have
\begin{align*}
e_{-\boldsymbol{\alpha}}(\boldsymbol{y})\frac{\partial(e_{\boldsymbol{\alpha}}(\boldsymbol{y})f(\boldsymbol{y}))}{\partial y_j}=-R_j^{\boldsymbol{\alpha}}(R_1^{\boldsymbol{\alpha}}-iR_2^{\boldsymbol{\alpha}})\Bigg(e_{-\boldsymbol{\alpha}}(\boldsymbol{y})\frac{\partial(e_{\boldsymbol{\alpha}}(\boldsymbol{y})f(\boldsymbol{y}))}{\partial y_1}+ie_{\boldsymbol{\alpha}}(\boldsymbol{y})\frac{\partial(e_{\boldsymbol{\alpha}}(\boldsymbol{y})f(\boldsymbol{y}))}{\partial y_2}\Bigg),
\end{align*}
for $j=1,2$ and   $\boldsymbol{\alpha}=(\alpha_1,\alpha_2)\in \mathbb{R}^2$ with $\alpha_k\notin \pi \mathbb{Z}, k=1,2  $.
\end{lemma}
\begin{proof}Taking the FRFT of the above identity,  we have
\begin{align*}
&\mathcal{F}_{\boldsymbol{\alpha}}\left(-R_j^{\boldsymbol{\alpha}}(R_1^{\boldsymbol{\alpha}}-iR_2^{\boldsymbol{\alpha}})\left(e_{-\boldsymbol{\alpha}}(\boldsymbol{y})\frac{\partial(e_{\boldsymbol{\alpha}}(\boldsymbol{y})f(\boldsymbol{y}))}{\partial y_1}+ie_{\boldsymbol{\alpha}}(\boldsymbol{y})\frac{\partial(e_{\boldsymbol{\alpha}}(\boldsymbol{y})f(\boldsymbol{y}))}{\partial y_2}\right)\right)(x)\\
=& \ -\frac{i\tilde{x}_j}{|\tilde{\boldsymbol{x}}|}\left(\frac{i\tilde{\boldsymbol{x}}_1}{|\tilde{\boldsymbol{x}}|}+\frac{i\tilde{\boldsymbol{x}}_2}{|\tilde{\boldsymbol{x}}|}\right)(ix_1\csc\alpha_1\mathcal{F}_{\boldsymbol{\alpha}} (f)(\boldsymbol{x})-x_2\csc\alpha_2\mathcal{F}_{\boldsymbol{\alpha}} (f)(\boldsymbol{x}))\\
=&\ -\frac{i\tilde{x}_j}{|\tilde{\boldsymbol{x}}|}\left(\frac{-|\tilde{x}_1|^2}{|\tilde{\boldsymbol{x}} |}+\frac{-|\tilde{x}_2|^2}{|\tilde{\boldsymbol{x}} |}\right)\mathcal{F}_{\boldsymbol{\alpha}} (f)(\boldsymbol{x})\\
=&\ ix_j\csc\alpha_j\mathcal{F}_{\boldsymbol{\alpha}} (f)(\boldsymbol{x}).
\end{align*}
By Lemma \ref{4th:1}, we have
\begin{align*}
\mathcal{F}_{\boldsymbol{\alpha}} \left(e_{-\boldsymbol{\alpha}}(\boldsymbol{y})\frac{\partial(e_{\boldsymbol{\alpha}}(\boldsymbol{y})f(\boldsymbol{y}))}{\partial y_j}\right)(\boldsymbol{x})=ix_j\csc\alpha_j\mathcal{F}_{\boldsymbol{\alpha}} (f)(\boldsymbol{x}).
\end{align*}
Applying the inverse FRFT on the  above identity  we deduce the desired result.
\end{proof}
Now, we  return to prove   Theorem \ref{4th:3}.
\begin{proof}By   Lemma \ref{4th:1} and Theorem \ref{2th:6}, we have
\begin{align*}
\left\|\frac{\partial(e_{\boldsymbol{\alpha}} f)}{\partial y_j}\right\|_{L^p}
=&\ \Bigg\|-R_j^{\boldsymbol{\alpha}}(R_1^{\boldsymbol{\alpha}}-iR_2^{\boldsymbol{\alpha}})\Bigg(e_{-\boldsymbol{\alpha}}(\boldsymbol{y})\frac{\partial(e_{\boldsymbol{\alpha}}(\boldsymbol{y})f(\boldsymbol{y}))}{\partial y_1} +ie_{\boldsymbol{\alpha}}(\boldsymbol{y})\frac{\partial(e_{\boldsymbol{\alpha}}(\boldsymbol{y})f(\boldsymbol{y}))}{\partial y_2}\Bigg)\Bigg\|_{L^p}\\
\leq&\ C\left\|e_{-\boldsymbol{\alpha}}(\boldsymbol{y})\frac{\partial(e_{\boldsymbol{\alpha}}(\boldsymbol{y})f(\boldsymbol{y}))}{\partial y_1} +ie_{\boldsymbol{\alpha}}(\boldsymbol{y})\frac{\partial(e_{\boldsymbol{\alpha}}(\boldsymbol{y})f(\boldsymbol{y}))}{\partial y_2}\right\|_{L^p},
\end{align*}
which completes the proof of the theorem.
\end{proof}
\begin{remark}
When $\boldsymbol{\alpha}=(\frac{\pi}{2}+k_1\pi,\frac{\pi}{2}+k_2\pi,\dots, \frac{\pi}{2}+k_n\pi), k_j\in \mathbb{Z} $ for $j=1,2,\dots,m$, Theorem \ref{4th:3}   simplifies to Proposition $4$   in \cite[pp.$60$]{s}. \end{remark}

\begin{lemma}\label{4th:6}
For $f\in S (\mathbb{R}^n)$ and $1\leq j,k\leq n $,  we have
$$ e_{ -\boldsymbol{\alpha}}(\boldsymbol{y})\frac{\partial^2 (e_{ \boldsymbol{\alpha}}(\boldsymbol{y})f(\boldsymbol{y}))}{\partial y_ky_j}=\left(-R^{\boldsymbol{\alpha}}_kR^{\boldsymbol{\alpha}}_je_{ -{\boldsymbol {\alpha}}}\triangle(e_{ \boldsymbol{\alpha}}f)\right)(\boldsymbol{y}),$$
 for all $\boldsymbol{y}\in \mathbb{R}^n$, where $\boldsymbol{\alpha}=(\alpha_1,\alpha_2,\dots,\alpha_n)\in \mathbb{R}^n$ with $\alpha_k\notin \pi \mathbb{Z}$.\end{lemma}
\begin{proof}   Taking the FRFT of the above identity,  we have
\begin{align*}
\mathcal{F}_{\boldsymbol {\alpha}} \left(e_{ -\boldsymbol {\alpha}}(\boldsymbol {y})\frac{\partial^2 \left(e_{ \boldsymbol {\alpha}}(\boldsymbol {y})f(\boldsymbol {y})\right)}{\partial y_k\partial y_j}\right)(\boldsymbol {x})
=&\ \mathcal{F}_{\boldsymbol {\alpha}} \left(e_{ -\boldsymbol {\alpha}}(\boldsymbol {y})\frac{\partial\left(e_{ \boldsymbol {\alpha}}(\boldsymbol {y})\frac{e_{ -\boldsymbol {\alpha}}(\boldsymbol {y})\partial(e_{\boldsymbol { \alpha}}(y)f(y))}{\partial y_j}\right)}{\partial y_k}\right)(\boldsymbol {x})\\
=&\ ix_k\csc\alpha_k  \mathcal{F}_{\boldsymbol {\alpha}}\left(e_{-{\boldsymbol {\alpha}}}(\boldsymbol {y})\frac{\partial[e_{ \boldsymbol {\alpha}}(\boldsymbol {y})f(\boldsymbol {y})]}{\partial y_j}\right)(\boldsymbol {x})\\
=&\ ix_k\csc\alpha_k  ix_j\csc\alpha_j\mathcal{F}_{\boldsymbol {\alpha}}(f)(\boldsymbol {x})\\
=&\ -\left(\frac{i\tilde{x}_k}{|\tilde{\boldsymbol {x}}|}\right)\left(\frac{i\tilde{x}_j}{|\tilde{\boldsymbol {x}}|}\right)(-|\tilde{\boldsymbol {x}}|^2)\mathcal{F}_{\boldsymbol {\alpha}}(f)(\boldsymbol {x})\\
=&\ -\left(\frac{i\tilde{x}_k}{|\tilde{\boldsymbol {x}}|}\right)\left(\frac{i\tilde{x}_j}{|\tilde{\boldsymbol {x}}|}\right)\mathcal{F}_\alpha\Bigg(e_{ \boldsymbol {-\alpha}}(\boldsymbol {y})\frac{\partial^2(e_{ \boldsymbol { \alpha}}(\boldsymbol {y})f(\boldsymbol {y}))}{\partial^2 y_1}+\cdots\\
&\ \ \ \ \ \ \ \ \ \ \ \ \ \ \ \ \ \ \ \ \ \ \ \ \ \ \ \ +e_{ \boldsymbol {-\alpha}}(\boldsymbol {y})\frac{\partial^2(e_{ \boldsymbol { \alpha}}(\boldsymbol {y})f(\boldsymbol {y}))}{\partial ^2y_n}\Bigg)(\boldsymbol {x})\\
=\ &\mathcal{F}_{\boldsymbol {\alpha}}(-R^{\boldsymbol {\alpha}}_kR^{\boldsymbol {\alpha}}_je_{ \boldsymbol {-\alpha}}\triangle(e_{ \boldsymbol {\alpha}}f))(\boldsymbol {x}).
\end{align*}
Applying the inverse FRFT on the  above identity, we obtain the desired result.
\end{proof}
\begin{remark}
When $\boldsymbol{\alpha}=(\frac{\pi}{2}+k_1\pi,\frac{\pi}{2}+k_2\pi,\dots, \frac{\pi}{2}+k_n\pi), k_j\in \mathbb{Z} $ for $j=1,2,\dots,m$, Lemma \ref{4th:6} can be simplified to Proposition $5.1.17$   in \cite{g1}. \end{remark}
\begin{theorem}\label{4th:8}
Suppose $f\in S (R^n)$ and $\Delta(e_{\boldsymbol{\alpha}} f)=\sum^n_{j=1}\frac{\partial^2(e_{\boldsymbol{\alpha}} f)}{\partial y_k\partial y_j}$. Then we have  a priori bound
$$\left\|\frac{\partial^2(e_{\boldsymbol{\alpha}} f)}{\partial y_k\partial y_j}\right\|_{L^p}\leq C \Big\|\Delta(e_{\boldsymbol{\alpha}} f)\Big\|_{L^p},$$
for $\boldsymbol{\alpha}=(\alpha_1,\alpha_2,\dots,\alpha_n)\in \mathbb{R}^n$ with $\alpha_k\notin \pi \mathbb{Z}, \ k=1,2,\dots,n.$\end{theorem}
\begin{proof}According to   Lemma \ref{4th:6} and   Theorem \ref{2th:6}, we  obtain that
\begin{align*}
\left\|\frac{\partial^2(e_{\boldsymbol{\alpha}} f)}{\partial y_k\partial y_j}\right\|_{L^p}=&\  \left\|e_{-\boldsymbol{\alpha}}\frac{\partial^2(e_{\boldsymbol{\alpha}} f)}{\partial y_k\partial y_j}\right\|_{L^p}\\
 =&\ \left\|-R_j^{\boldsymbol{\alpha}} R_k^{\boldsymbol{\alpha}} e_{-\boldsymbol{\alpha}}\triangle(e_{\boldsymbol{\alpha}} f)\right\|_{L^p}\\
 \leq &\ C\|\triangle(e_{\boldsymbol{\alpha}} f)\|_{L^p},
\end{align*}
which completes the proof of the theorem.
\end{proof}

\subsection{A characterization of Laplace's equation }

\ \ \ \  Next we give a characterization of Laplace's equation.
 \begin{example}Suppose that  $u\in S '(\mathbb{R}^n)$   and   $f\in L^2(\mathbb{R}^n)$.
  We solve   Laplace's equation 
  $$\Delta(e_{\boldsymbol{\alpha}} u)=e_{{\boldsymbol{\alpha}}} f. \eqno(4.1)\label{4.1}$$
 \end{example}
One of the methods can be found in  \cite{cfgw}. In this paper, we express all second-order derivatives of $u$ in terms of the fractional Riesz transform of $f$. In order to accomplish  this we provide  a necessary lemma.

\begin{lemma}\label{4th:7}\emph{(\cite{g1})}
Suppose that $u\in S '(\mathbb{R}^n)$. If $\hat{u}$ is supported  at  $\{0\}$, then $u$ is a polynomial.
\end{lemma}

 To solve equation (\ref{4.1}), we first show that  the tempered distribution $$\mathcal{F}_{\boldsymbol{\alpha}}(e_{-\boldsymbol{\alpha}}\partial_j\partial_k(e_{\boldsymbol{\alpha}}u)+R_j^{\boldsymbol{\alpha}} R_k^{\boldsymbol{\alpha}} f) $$ is supported at $\{0\}$. From      Lemma \ref{4th:7}, we can obtain that $$e_{ -\boldsymbol{\alpha}}\partial_j\partial_k(e_{ \boldsymbol{\alpha}}u)=-R_j^{\boldsymbol{\alpha}} R_k^{\boldsymbol{\alpha}} f+e_{ -\boldsymbol{\alpha}}P,$$
where $P$ is a polynomial of $n$ variables (that depends on $j$ and $k$). Then we   provide a method of expressing mixed partial derivatives of $e_{ \alpha}u$ in terms of the fractional Riesz transform of $f$.

 To prove that  the tempered distribution $\mathcal{F}_\alpha(e_{ -\alpha}\partial_j\partial_k(e_{ \alpha}u)+R_j^\alpha R_k^\alpha f) $ is supported at $\{0\}$,   we pick $\gamma\in S (\mathbb{R}^n)$ whose support does not contain the origin.  Then, $\gamma$ vanishes in a neighborhood of zero.  Fix $\eta \in C^\infty$,  which is equal to 1 on the support of $\gamma$ and vanishes in a smaller neighborhood of zero.  Define
$$\zeta(\boldsymbol{\xi})=-\eta(\boldsymbol{\xi})\left(-\frac{i\tilde{\xi}_j}{|\tilde{\boldsymbol{\xi}}|}\right)\left(-\frac{i\tilde{\xi}_k}{|\tilde{\boldsymbol{\xi}}|}\right)$$
and we notice that $\zeta$  and  all of  its derivatives are both  bounded $C^\infty$  functions. Additionally
$$\eta(\boldsymbol{\xi})(i\tilde{\xi}_j)(i\tilde{\xi}_k)=\zeta(\boldsymbol{\xi})(-|\tilde{\boldsymbol{\xi}}|^2).$$

Taking the FRFT of both side of (\ref{4.1}) we   obtain that
$$\mathcal{F}_{\boldsymbol{\alpha}}(e_{-\boldsymbol{\alpha}} \Delta(e_{\boldsymbol{ \alpha}} u))(\boldsymbol{\xi})=-|\tilde{\boldsymbol{\xi}}|^2\mathcal{F}_{\boldsymbol{\alpha}}(u)(\boldsymbol{\xi})=\mathcal{F}_{\boldsymbol{\alpha}}(f)(\boldsymbol{\xi}).$$
Multiplying by $\zeta$, we write 
$$\zeta(\boldsymbol{\xi})\mathcal{F}_{\boldsymbol{\alpha}}(e_{-\boldsymbol{\alpha}}\Delta(e_{ \boldsymbol{\alpha}} u))(\boldsymbol{\xi})=-\zeta(\boldsymbol{\xi})|\tilde{\boldsymbol{\xi}}|^2\mathcal{F}_{\boldsymbol{\alpha}}(u)(\boldsymbol{\xi})=\zeta(\boldsymbol{\xi})\mathcal{F}_\alpha(f)(\boldsymbol{\xi}).$$
Since for all $1\leq j,k \leq n$,
\begin{align*}
\langle \mathcal{F}_{\boldsymbol{\alpha}}(e_{ -\boldsymbol{\alpha}}\partial_j\partial_k(e_{ \boldsymbol{\alpha}}u)),\gamma\rangle=&\ \langle(i\tilde{\xi}_j)(i\tilde{\xi}_k)\mathcal{F}_{\boldsymbol{\alpha}}(u),\gamma\rangle\\
=&\ \langle(i\tilde{\xi}_j)(i\tilde{\xi}_k)\mathcal{F}_\alpha(u),\eta\gamma\rangle\\
=&\ \langle\eta(\boldsymbol{\xi})(i\tilde{\xi}_j)(i\tilde{\xi}_k)\mathcal{F}_{\boldsymbol{\alpha}}(u),\gamma\rangle\\
=&\ \langle\zeta(\boldsymbol{\xi})(-|\tilde{\boldsymbol{\xi}}|^2)\mathcal{F}_{\boldsymbol{\alpha}}(u),\gamma\rangle\\
=&\ \langle\zeta(\boldsymbol{\xi})\mathcal{F}_{\boldsymbol{\alpha}}(f),\gamma\rangle\\
=&\ \left\langle-\eta(\boldsymbol{\xi})\left(- i\tilde{\xi}_j/|\tilde{\boldsymbol{\xi}}| \right)\left(- i\tilde{\xi}_k/|\tilde{\boldsymbol{\xi}}|\right) \mathcal{F}_{\boldsymbol{\alpha}},\gamma\right\rangle\\
=&\ \langle-\eta(\boldsymbol{\xi})  \mathcal{F}_{\boldsymbol{\alpha}}(R_j^{\boldsymbol{\alpha}} R_k^{\boldsymbol{\alpha}} (f)),\gamma\rangle\\
=&\ -\langle  \mathcal{F}_{\boldsymbol{\alpha}}(R_j^{\boldsymbol{\alpha}} R_k^{\boldsymbol{\alpha}} (f)),\eta \gamma\rangle\\
=&\ -\langle  \mathcal{F}_{\boldsymbol{\alpha}}(R_j^{\boldsymbol{\alpha}} R_k^{\boldsymbol{\alpha}} (f)), \gamma\rangle,
\end{align*}
and since  the support   of $\gamma\in S (\mathbb{R}^n)$ does not contain  the  origin, 
it follows that the function 
$\mathcal{F}_{\boldsymbol{\alpha}}(e_{-\boldsymbol{\alpha}}\partial_j\partial_k(e_{\boldsymbol{\alpha}}u)+R_j^{\boldsymbol{\alpha}}R_k^{\boldsymbol{\alpha}}( f))$
is supported at $\{0\}$.

\section{Numerical simulation of  fractional multipliers}

\label{sect:ex}
\ \ \ \ In this section, we apply the fractional Riesz transform to an image with the help of the FRFT discrete
algorithm (\cite{bm1,bm2,tlx}).

\begin{figure}[H]
	\centering
		\subfigure[]{\includegraphics[width=0.4\linewidth]{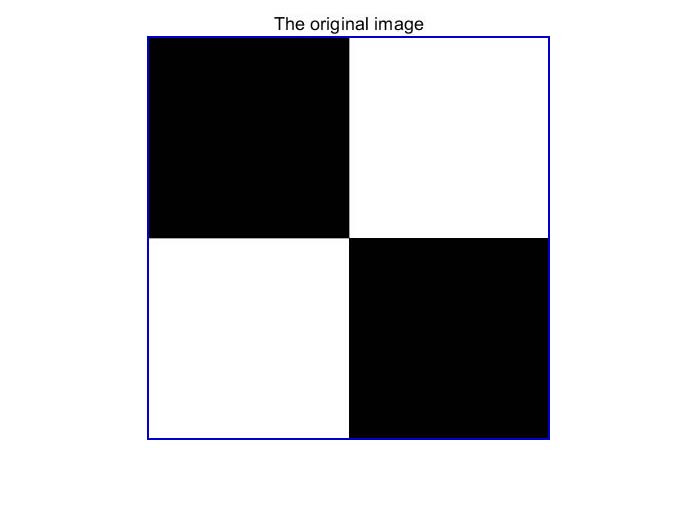} }
		\subfigure[]{\includegraphics[width=0.4\linewidth]{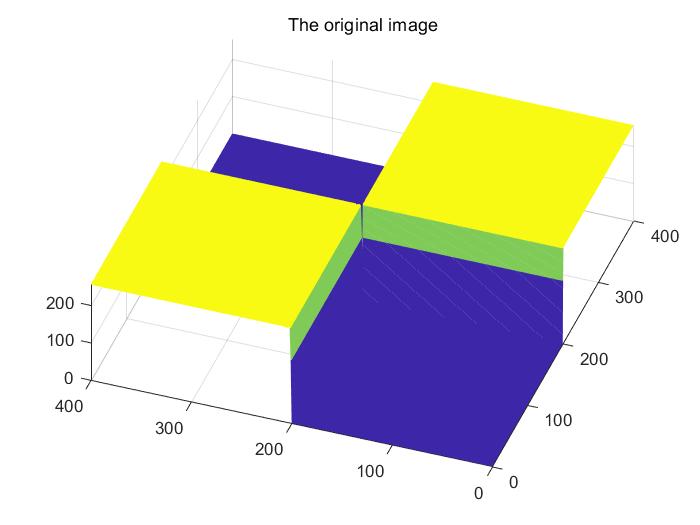} }
		\subfigure[]{\includegraphics[width=0.4\linewidth]{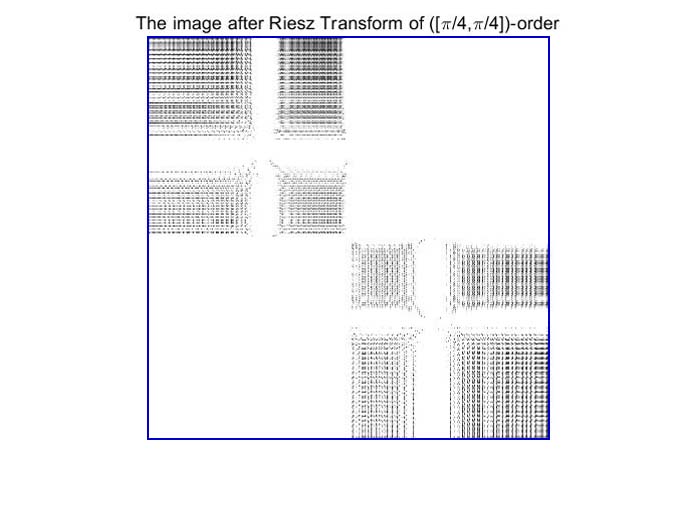} }
		\subfigure[]{\includegraphics[width=0.4\linewidth]{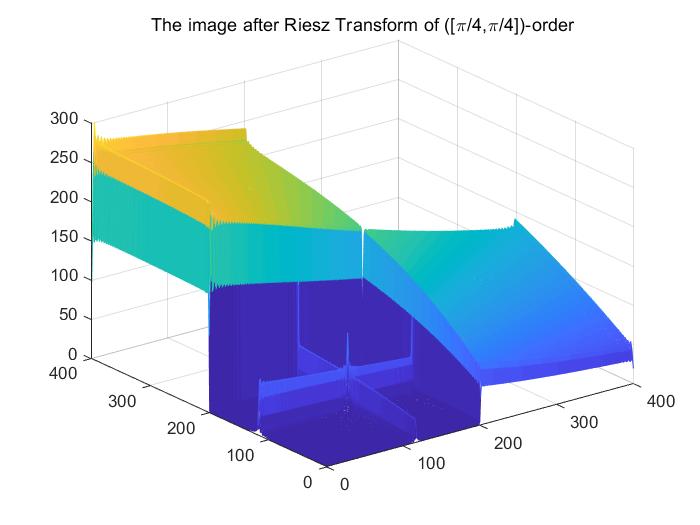} }
    \caption{Fractional Riesz transform of order $(\pi/4, \pi/4)$ of a image.  }
    \label{fig:image}%
\end{figure}

As shown in Fig.
\ref{fig:image}, (a) is the original 2-dimensional   grayscale image with 400 pixels $\times$ 400 pixels; (c) is
the 2-dimensional   grayscale image after the  Riesz transform of order $(\pi/4,\pi/4)$.
In the continuous case, Fig.
\ref{fig:image} (a)  can be regarded as  a function of $\mathbb{R}^2$
\[f(x_1,x_2)=
   \left\{
		\begin{array}[c]{ll}
		0, & (x_1,x_2) \in [0,200]^2 \cup [200,400]^2 , \\
		255, & \textup{otherwise.}
		\end{array}
   \right.
\]
 Fig. \ref{fig:image} (b) and (d) are the 3-dimensional color graphs of $f$ and $R_{x_1}^{(\pi/4,\pi/4)}f$.
 
Recall  that the fractional Fourier multiplier of the fractional Riesz transform
$R_j^{\boldsymbol{\alpha}}$ is
\[
m_j^{\boldsymbol{\alpha}} (\boldsymbol{u}):=  -i \frac{\tilde{u}_j}{|\tilde{\boldsymbol{u}}|}.
 \]
 Graphs (a) and (b) in Fig. \ref{fig:phase} indicate that the fractional Riesz transform has the
effect of amplitude  reduction.

By  comparing  (c)/(e) and (d)/(f) in Fig. \ref{fig:phase}, as well as the real/imaginary part of
$\mathcal{F}_{\alpha} f$  and the imaginary/real part of $\mathcal{F}_{\boldsymbol{ \alpha}} (R_j^{\boldsymbol{ \alpha}}f)$
accordingly, it can be seen   that the
fractional Riesz transform has the effect of phase shifting.

Above all, Fig. \ref{fig:phase} shows that the fractional multiplier of $R_j^{\boldsymbol{\alpha}}$ is   $-i \tilde{u}_j/|\tilde{\boldsymbol{u}}|$, and thus,  the fractional Riesz transform is not only a phase-shift converter but also an amplitude attenuator.  This circumstance is quite different from that of  the  fractional  multiplier of $H_{\boldsymbol{\alpha}}$, which is only    a  phase-shift converter  with multiplier $-i{\rm sgn}((\pi-\alpha)u)$.

 \begin{figure}[H]
	\centering
	\subfigure[]{\includegraphics[width=0.4\linewidth]{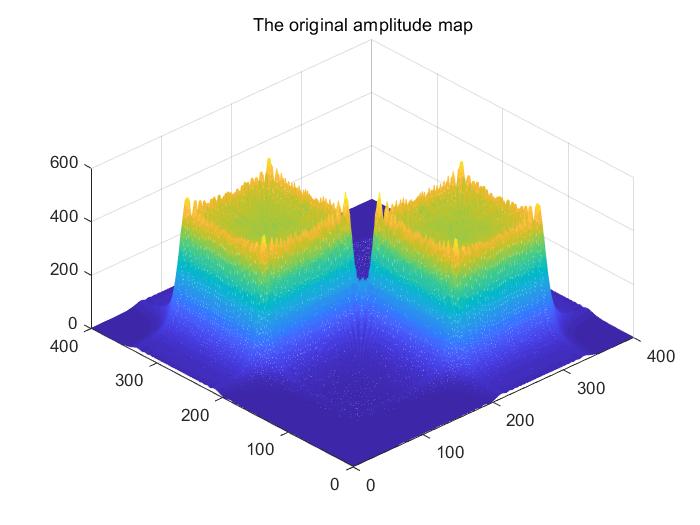} }
	\subfigure[]{\includegraphics[width=0.4\linewidth]{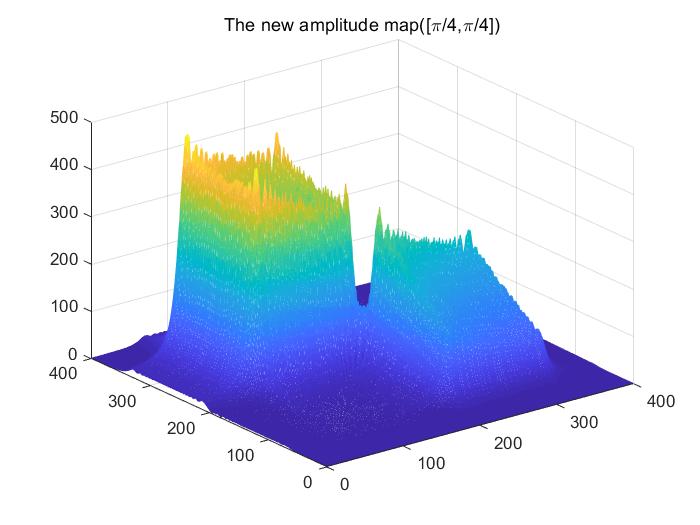} }
	\subfigure[]{\includegraphics[width=0.4\linewidth]{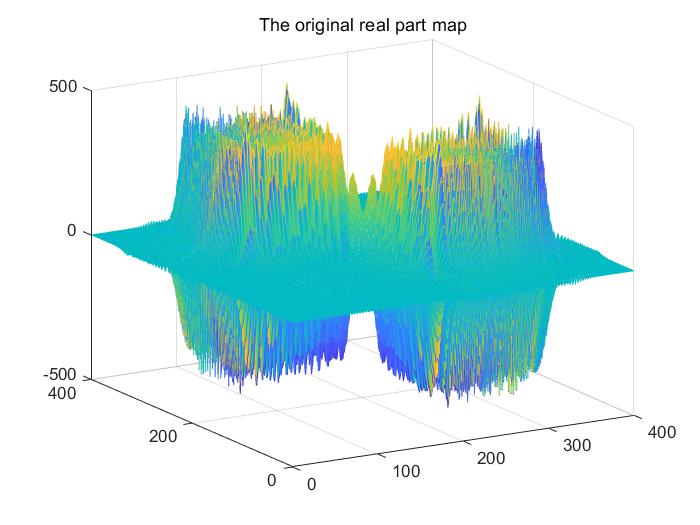} }
	\subfigure[]{\includegraphics[width=0.4\linewidth]{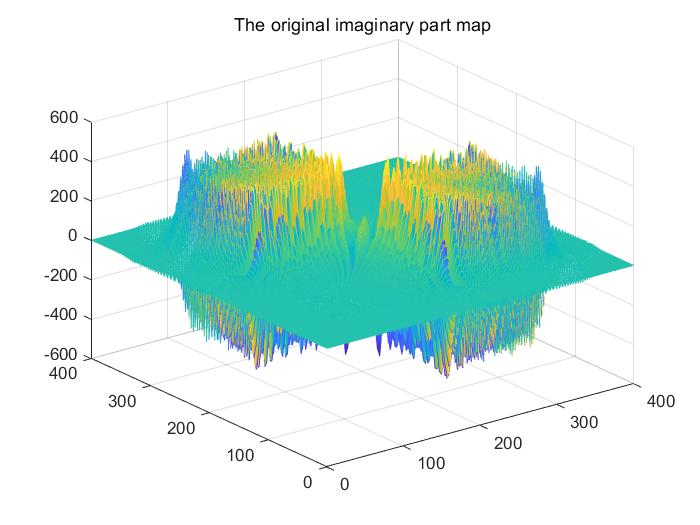} }
	\subfigure[]{\includegraphics[width=0.4\linewidth]{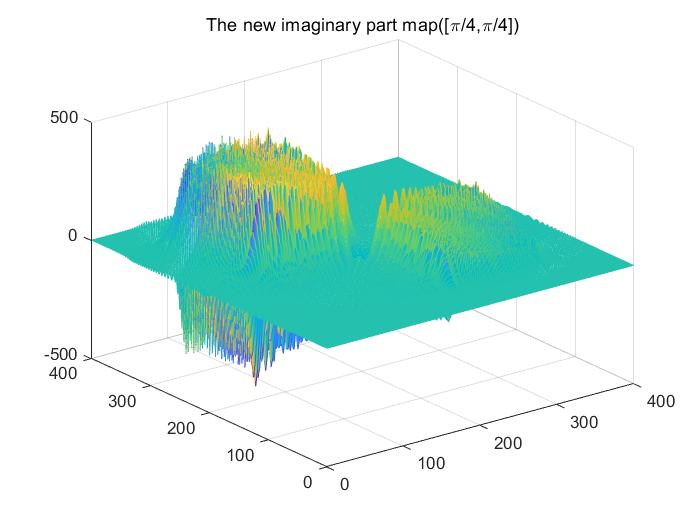} }
	\subfigure[]{\includegraphics[width=0.4\linewidth]{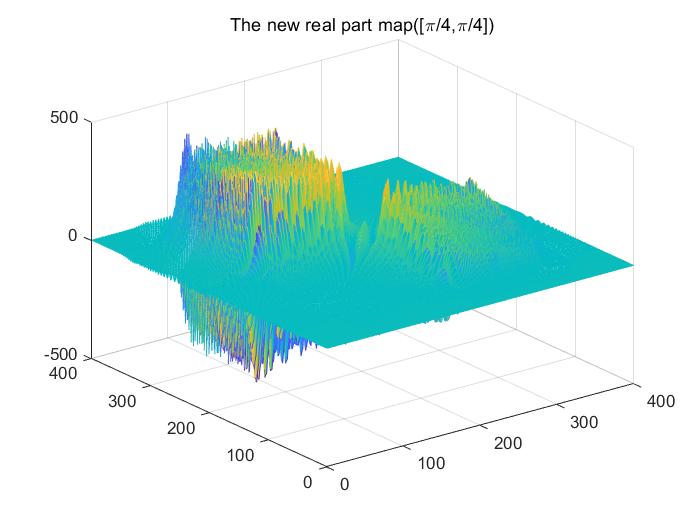} }
	    \caption{Phase-shifting and amplitude-reducing effect of the fractional Riesz transform in
	    the fractional Fourier domain of order $\boldsymbol{\alpha}=(\pi/4,\pi/4)$ on an image. }
	\label{fig:phase}%
\end{figure}

\section{Application of the fractional Riesz transform in edge detection}
\label{sect:ex}

\ \ \ \ Edge detection is a key technology in image processing. It is widely used in biometrics, image understanding, visual attention and other fields. Commonly used image feature extraction methods include the Roberts operator, Prewitt operator, Sobel operator, Laplacian operator and Canny operator. These algorithms extract features based on the gradient changes in the    pixel amplitudes. In  \cite{la,zl,zzm}, the authors introduced the  image edge detection methods based on the Riesz transform, which can avoid the influence of uneven illumination. Moreover, the Riesz transform has the characteristics of isotropy; therefore, the Riesz transform has more advantages in feature extraction. Based on the principle of  Riesz transform edge detection,  edge detection  based on the \emph{fractional Riesz transform} is proposed in this section.

When processing the two-dimensional signal $f(\boldsymbol{x})$, the fractional Riesz transform of $f(\boldsymbol{x})$ can be expressed as
\begin{align}\label{6.1}
(R_1^{\boldsymbol{\alpha}}f) (\boldsymbol{x}) =c_n \ {\rm p}.{\rm v}. \ e_{-\boldsymbol{\alpha}}(\boldsymbol{x})\left((e_{\boldsymbol{\alpha}}f)\ast \frac{x_1}{|\boldsymbol{x}|^3}\right) (\boldsymbol{x}),\\ \label{6.2}
(R_2^{\boldsymbol{\alpha}}f) (\boldsymbol{x}) =c_n \ {\rm p}.{\rm v}. \ e_{-\boldsymbol{\alpha}}(\boldsymbol{x})\left((e_{\boldsymbol{\alpha}}f)\ast \frac{x_2}{|\boldsymbol{x}|^3}\right) (\boldsymbol{x}),
\end{align}
 or
\begin{align}\label{6.3}
(R_1^{\boldsymbol{\alpha}}f) (\boldsymbol{x}) =\left[ \mathcal{F}_{-\boldsymbol{\alpha}}\left( -i
\frac{u_1\csc \alpha_1}{\sqrt{(u_1\csc \alpha_1)^2+(u_2\csc \alpha_2)^2}} (\mathcal{F}_{\boldsymbol{\alpha}} f) (\boldsymbol{u}) \right)\right] (\boldsymbol{x}),\\ \label{6.4}
(R_2^{\boldsymbol{\alpha}}f) (\boldsymbol{x}) =\left[ \mathcal{F}_{-\boldsymbol{\alpha}}\left( -i
\frac{u_2\csc \alpha_2}{\sqrt{(u_1\csc \alpha_1)^2+(u_2\csc \alpha_2)^2}} (\mathcal{F}_{\boldsymbol{\alpha}} f) (\boldsymbol{u}) \right)\right] (\boldsymbol{x}).
\end{align}

For an image $f(\boldsymbol{x})$, the monogenic signal is defined as the combination of $f(\boldsymbol{x})$ and its fractional Riesz transform.
$$(p(\boldsymbol{x}),q_1(\boldsymbol{x}),q_2(\boldsymbol{x}))=(f(\boldsymbol{x}), (R_1^{\boldsymbol{\alpha}}f) (\boldsymbol{x}),  (R_2^{\boldsymbol{\alpha}}f) (\boldsymbol{x})).$$
Therefore, the local amplitude value $A(\boldsymbol{x})$, local   orientation $\theta(\boldsymbol{x})$ and local phase $P(\boldsymbol{x})$ in the monogenic signal in the image can be expressed as
\begin{eqnarray*}
\left\{
\begin{aligned}
A(x)&=\sqrt{p(\boldsymbol{x})^2+|q_1(\boldsymbol{x})|^2+|q_2(\boldsymbol{x})|^2}\\
\theta(x)&=\tan^{-1}\left(\left|\frac{q_2(\boldsymbol{x})}{q_1(\boldsymbol{x})}\right|\right)\\
P(x)&=\tan^{-1}\left(\frac{p(\boldsymbol{x})}{\sqrt{|q_1(\boldsymbol{x})|^2+|q_2(\boldsymbol{x})|^2}}\right).
\end{aligned}
\right.
\end{eqnarray*}

In this paper, we use the fractional Riesz transform in the form of (\ref{6.3}) and (\ref{6.4}) for algorithm design because  the fractional Riesz transform can be decomposed into a combination of a FRFT, inverse FRFT and multiplier by the fractional multiplier Theorem \ref{2th:1}. Because the FRFT and inverse FRFT have  fast algorithms, compared with the   form  of (\ref{6.1}) and (\ref{6.2}), the computational complexity   of the  algorithm in the form of (\ref{6.3}) and (\ref{6.4}) is  reduced. Thus, a faster and more  efficient operation is realized.

\begin{figure}[H]
	\centering
	\subfigure[]{\includegraphics[width=0.4\linewidth]{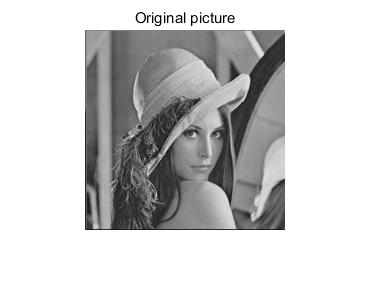} }
	\subfigure[]{\includegraphics[width=0.4\linewidth]{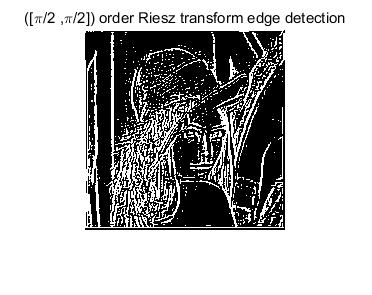} }
	    \caption{Original image and edge detection based on the fractional   Riesz transform of order $([\pi/2, \pi/2])$  (i.e., the classical Riesz transform) }
	\label{fig:6.1}%
\end{figure}

\begin{figure}[H]
	\centering
	\subfigure[]{\includegraphics[width=0.3\linewidth]{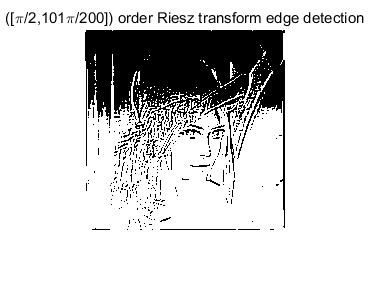} }
	\subfigure[]{\includegraphics[width=0.3\linewidth]{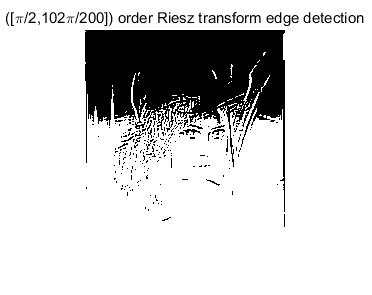} }
	\subfigure[]{\includegraphics[width=0.3\linewidth]{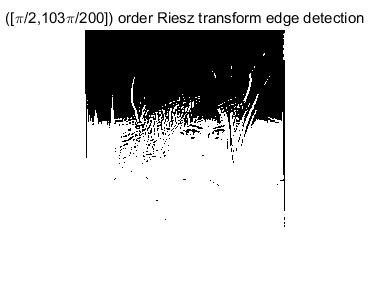} }
	\subfigure[]{\includegraphics[width=0.3\linewidth]{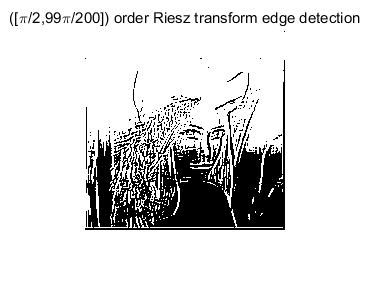} }
	\subfigure[]{\includegraphics[width=0.3\linewidth]{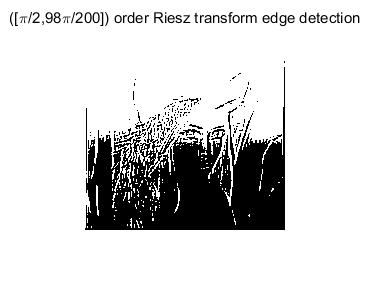} }
	\subfigure[]{\includegraphics[width=0.3\linewidth]{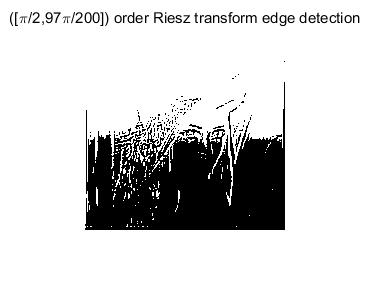} }
	    \caption{Extract the information of a specific position in the lateral direction of the image }
	\label{fig:6.2}%
\end{figure}

\begin{figure}[H]
	\centering
	\subfigure[]{\includegraphics[width=0.3\linewidth]{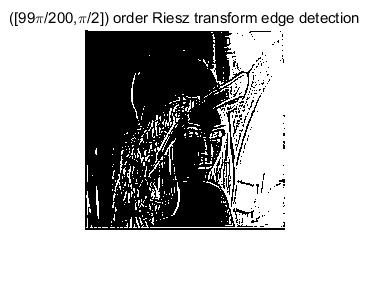} }
	\subfigure[]{\includegraphics[width=0.3\linewidth]{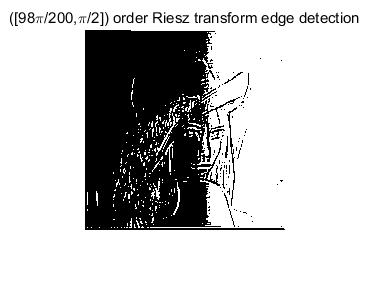} }
	\subfigure[]{\includegraphics[width=0.3\linewidth]{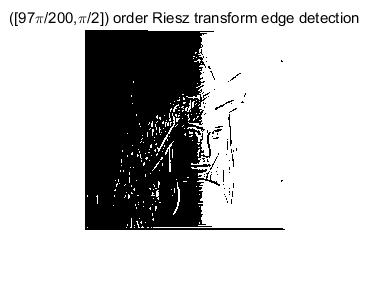} }
	\subfigure[]{\includegraphics[width=0.3\linewidth]{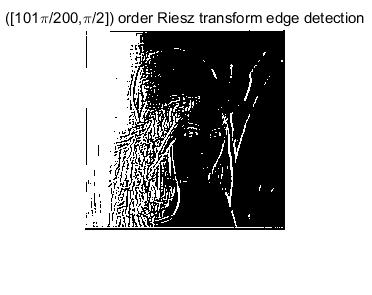} }
	\subfigure[]{\includegraphics[width=0.3\linewidth]{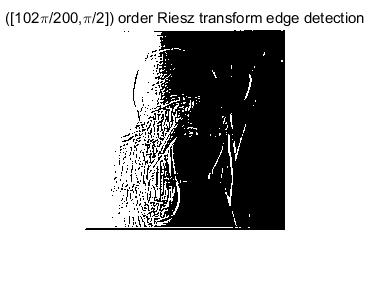} }
	\subfigure[]{\includegraphics[width=0.3\linewidth]{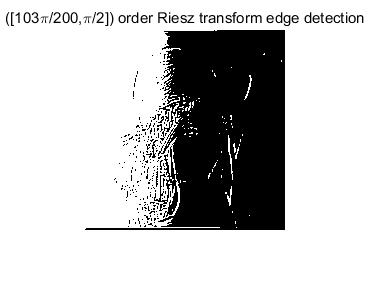} }
	    \caption{Extract the information of specific position in the longitudinal direction of the image}
	\label{fig:6.3}%
\end{figure}

 \begin{figure}[H]
	\centering
	\subfigure[]{\includegraphics[width=0.3\linewidth]{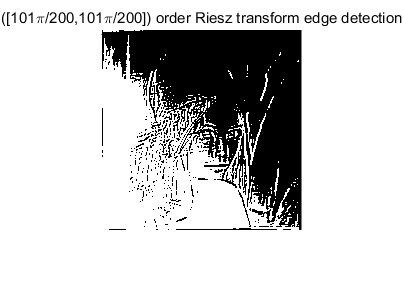} }
	\subfigure[]{\includegraphics[width=0.3\linewidth]{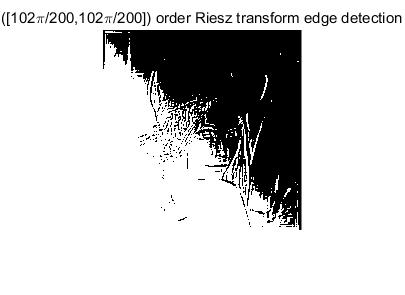} }
	\subfigure[]{\includegraphics[width=0.3\linewidth]{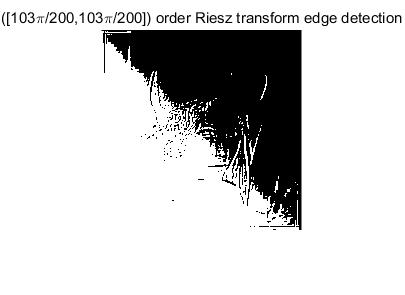} }
	\subfigure[]{\includegraphics[width=0.3\linewidth]{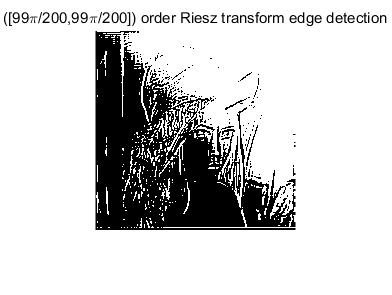} }
	\subfigure[]{\includegraphics[width=0.3\linewidth]{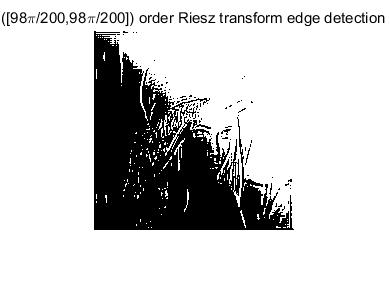} }
	\subfigure[]{\includegraphics[width=0.3\linewidth]{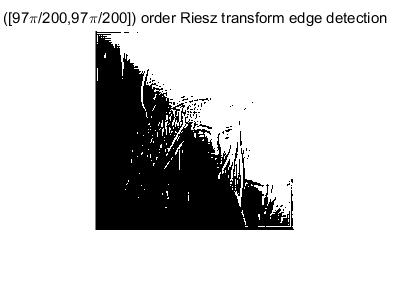} }
	    \caption{Extract the information of the specific position in the main diagonal of the image }
	\label{fig:6.4}%
\end{figure}

\begin{figure}[H]
	\centering
	\subfigure[]{\includegraphics[width=0.3\linewidth]{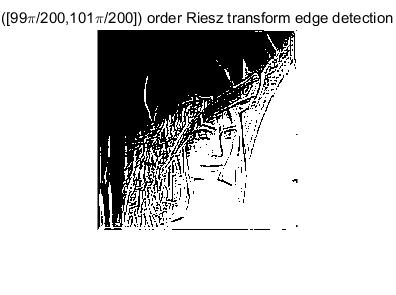} }
	\subfigure[]{\includegraphics[width=0.3\linewidth]{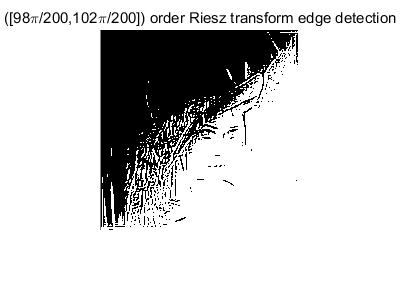} }
	\subfigure[]{\includegraphics[width=0.3\linewidth]{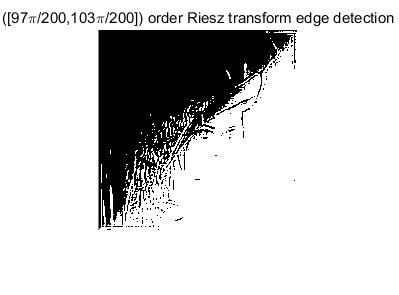} }
	\subfigure[]{\includegraphics[width=0.3\linewidth]{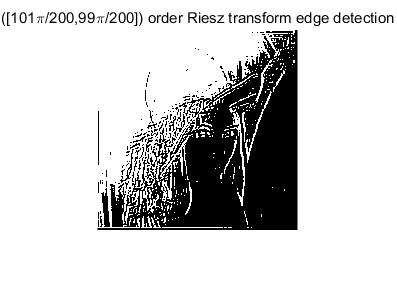} }
	\subfigure[]{\includegraphics[width=0.3\linewidth]{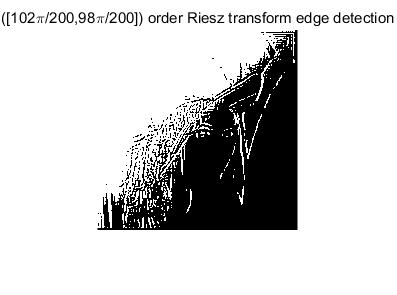} }
	\subfigure[]{\includegraphics[width=0.3\linewidth]{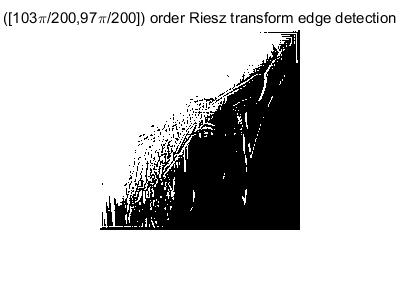} }
	    \caption{Extract the information of the specific position in the counter  diagonal of the image}
	\label{fig:6.5}%
\end{figure}

The simulation experiment will be  conducted    by using  the classical Lena image ((a) in Fig. \ref{fig:6.1}) as the test image. To better highlight the results of our simulation experiment, we choose an appropriate threshold value to binarize the images after the experiment in such a way  that our experimental results show a more obvious effect. Graph (b) in Fig. \ref{fig:6.1}  illustrates the result of edge detection based on the  fractional   Riesz transform of order $([\pi/2, \pi/2])$ (i.e., the classical Riesz transform).

Fig. \ref{fig:6.2} shows that when $\alpha_1 = \pi/2$ is fixed, if $\alpha_2$   decreases from $\pi/2$ ((d), (e), (f) in Fig. \ref{fig:6.2}), we extract information about the lateral up position. If $\alpha_2$   increases from $\pi/2$ ((a), (b), (c) in Fig. \ref{fig:6.2}), we   extract information about the lateral down position.  In conclusion,  by  fixed $\alpha_1 = \pi/2$, we can adjust $\alpha_2$ to extract   information on the specific position in the lateral direction.

Fig. \ref{fig:6.3} indicates  that when $\alpha_2= \pi/2$ is fixed and $\alpha_1$   decreases from $\pi/2$ ((a), (b), (c) in Fig. \ref{fig:6.3}), the longitudinal right position information is extracted. As $\alpha_1$   increases from $\pi/2$ ((d), (e), (f) in Fig. \ref{fig:6.3}), we extract the longitudinal left position information. In conclusion, when  fixing $\alpha_2 = \pi/2$, we can extract information of the specific longitudinal positions by adjusting $\alpha_1$.

Fig. \ref{fig:6.4} shows that when $\alpha_1$ and $\alpha_2$  increase or decrease the same size from $\pi/2$,    the  information on  the specific direction in  the main   diagonal is extracted.

Fig. \ref{fig:6.5} indicates that when $\alpha_1$   increases from $\pi/2$  and $\alpha_2$  decreases from $\pi/2$  by the same size,
the  information on  the specific direction in  the   antidiagonal is extracted.

The preceding simulation provides  a new edge detection tool  based on the 
  the fractional Riesz transform,   that   extracts both global features and  local features of images. 
  This numerical implementation  confirms 
the belief expressed in \cite{xxwqwy} that   amplitude, phase, and direction information can be   simultaneously extracted   by controlling the order of the fractional Riesz transform.     We  predict  that  very comprehensive  analysis and processing of multidimensional   signals, such as images, videos, 3D meshes and animations, can be achieved via the fractional Riesz transforms.

\section{Conclusions}
\label{sect:ex}
\ \ \ \ In this paper, we introduced the fractional Riesz transform and give the corresponding fractional multiplier theorem and its applications in partial differential equations. We also studied   properties of chirp singular integral operators, chirp Hardy spaces and chirp BMO spaces. 
We used the fractional Riesz transforms in concrete applications in  edge detection with 
surprisingly good results. Our experiments indicate that edge detection can  extract local  information in any direction by adjusting the order of the fractional Riesz transform. The algorithm complexity of  the fractional Riesz transform in the form of (\ref{6.3}) and (\ref{6.4}) is  reduced compared with the fractional Riesz transform of  (\ref{6.1}) and (\ref{6.2}).

\end{document}